\documentclass[11pt]{amsart} 
\usepackage[bottom=1in]{geometry}
\usepackage{esint} 
\usepackage{graphicx, amssymb, hyperref }
\usepackage[mathscr]{euscript}
\usepackage{todonotes}
\usepackage[english]{babel}
\usepackage[utf8]{inputenx}

\newtheorem{theorem}{Theorem}[section]
\newtheorem{definition}[theorem]{Definition}
\newcommand{\ssize}{\text{size}\,}
\newcommand{\eenergy}{\text{energy}\,}
\newtheorem{lemma}[theorem]{Lemma}

\newtheorem{proposition}[theorem]{Proposition}

\newcommand{\sssize}{\widetilde{\text{size}}}

\newcommand{\one}{\mathbf{1}}
\newcommand{\dist}{\,\text{dist}\,}
\newcommand{\rr}{\mathbb}
\newcommand{\ii}{\mathscr}
\newcommand{\ci}{\tilde{\chi}}

\newcommand{\ds}{\displaystyle}

\newcommand{\cc}{\mathcal{C}}
\newcommand{\lft}{\left|}
\newcommand{\rg}{\right|}

\newtheoremstyle{dotless}{}{}{\itshape}{}{\bfseries}{}{ }{}
\theoremstyle{dotless}
\newtheorem*{remark}{Remark:}

\author{Cristina Benea}
\address{Cristina Benea, CNRS - Universit\'{e} de Nantes, Laboratoire Jean Leray, Nantes 44322, France}
\email{cristina.benea@univ-nantes.fr}

\author{Fr\'{e}d\'{e}ric Bernicot}
\address{Fr\'{e}d\'{e}ric Bernicot, CNRS - Universit\'{e} de Nantes, Laboratoire Jean Leray, Nantes 44322, France}
\email{frederic.bernicot@univ-nantes.fr}

\date{\today}

\keywords{Bilinear Fourier multipliers, orthogonality}
\subjclass{ 42B20, 42B25}

 \title{A Bilinear Rubio de Francia Inequality for Arbitrary Squares
} 
\begin{document}
\begin{abstract}
We prove the boundedness of a smooth bilinear Rubio de Francia operator associated with an arbitrary collection of squares (with sides parallel to the axes) in the frequency plane
\[
\left(f, g\right)\mapsto \left( \sum_{\omega \in \Omega}\lft\int_{\rr{R}^2} \hat{f}(\xi) \hat{g}(\eta) \Phi_{\omega}(\xi, \eta) e^{2 \pi i x\left(\xi+\eta \right)} d \xi d \eta\rg^r \right)^{1/r},
\] 
provided $r>2$. More exactly, we show that the above operator maps $L^p \times L^q \to L^s$ whenever $p, q, s'$ are in the ``local $L^{r'}$" range, i.e.  $\ds \frac{1}{p}+\frac{1}{q}+\frac{1}{s'}=1$, $\ds 0 \leq \frac{1}{p}, \frac{1}{q} <\frac{1}{r'}$, and $\ds\frac{1}{s'}<\frac{1}{r'}$. Note that we allow for negative values of $s'$, which correspond to quasi-Banach spaces $L^s$.
\end{abstract}

\maketitle

\section{Introduction}
\label{sec:intro}

The classical Littlewood-Paley theory states that the $L^p$ norm of a function is equivalent to the $L^p$ norm of the square function associated with (smooth) Fourier projection onto the dyadic intervals $\ds \left[ 2^j, 2^{j+1} \right]$:
\begin{equation}
\label{eq:littlewood-paley}
 \| f\|_p \leq c_p \left\| \left( \sum_j \lft f \ast \psi_j  \rg^2  \right)^{1/2}  \right\|_p \leq C_p \|f\|_p,
\end{equation}
for any $1<p< \infty$.

For an arbitrary sequence of disjoint intervals, we can only recover the RHS inequality of \eqref{eq:littlewood-paley} for $2 \leq p < \infty$, and consequently only the LHS for $1<p \leq 2$. Rubio de Francia proved in \cite{RF} that for any arbitrary collection of mutually disjoint intervals $\ds  \left[ a_k, b_k\right]$, the operator 
\begin{equation}
\label{def_RF_classic}
RF(f)(x)=\left( \sum_{ k } \lft \int_{\rr{R}} \hat{f}(\xi) \one_{\left[ a_k, b_k\right] }(\xi) e^{2 \pi i \xi x}    d \xi \rg^{2} \right)^{1/2},
\end{equation}
maps $L^p$ into $L^p$ boundedly, for any $p \geq 2$. We can regard this as a Littlewood-Paley inequality associated to \emph{disjoint, arbitrary Fourier projections}.

In higher dimensions, a similar result was proved by Journ\'{e} in \cite{JournCZopRF2}: given an arbitrary collection of mutually disjoint rectangles $\ds \lbrace R \rbrace_{R \in \ii{R}}$ in $\rr{R}^n$ with sides parallel to the axes, the operator
\begin{equation}
\label{def:RF-2d}
F \mapsto \left(  \sum_{R \in \ii{R}} \left| \int_{\rr{R}^n} \hat{F}(\xi_1, \ldots, \xi_n) \one_{R}(\xi_1, \ldots, \xi_n) e^{2 \pi i \left(\xi_1, \ldots , \xi_n\right) \cdot \left( x_1, \ldots , x_n \right)}    d \xi_1 \ldots d \xi_n\right|^{2} \right)^{1/2}
\end{equation}
maps $\ds L^p\left(\rr{R}^n\right)  \to L^p\left(\rr{R}^n\right)$ boundedly, for any $p \geq 2$.

A similar generic orthogonality principle for bilinear operators doesn't exist, except for some particular situations. More exactly, consider a family of bilinear operators $T_k$ associated with multipliers $m_k$:
\[
T_k(f, g)(x)=\int_{\rr{R}^2} \hat{f}(\xi) \hat{g}(\eta) m_k(\xi, \eta) e^{2 \pi i x\left(\xi+\eta\right)} d\xi d \eta.
\]
Assume that the $m_k$ have mutually disjoint supports in the frequency plane, and the operators $T_k$ are uniformly bounded within some range. What extra conditions should the $m_k$ satisfy in order to obtain
\begin{equation}
\label{eq:bil-princip}
\Big \|\left( \sum_k \big \vert T_k (f,g)\big \vert^2 \right)^{1/2}   \Big \|_s \lesssim \|f\|_p \|g\|_q,
\end{equation}
for some triple $(p, q, s)$ satisfying $\ds \frac{1}{p}+\frac{1}{q}=\frac{1}{s}$?
\medskip
~\\

Below we present a few examples from the existing literature of such square functions associated to bilinear operators $T_k$.

The natural bilinear version of \eqref{def:RF-2d} is the following operator:
\begin{equation}
\label{def:T_sharp}
T_{sharp}(f, g)(x):=\left( \sum_{\omega \in \Omega} \lft  \int_{\rr{R}^2} \hat{f}(\xi) \hat{g}(\eta) \one_{\omega}(\xi, \eta) e^{2 \pi i x \left( \xi+ \eta \right)} d \xi d \eta  \rg^2   \right)^{1/2},
\end{equation}
where $\Omega$ is an arbitrary collection of mutually disjoint squares, with sides parallel to the axes. We restrict our attention to squares in order to make sure that $T_{sharp}$ defined above is a one-parameter operator. It is not known if this operator is bounded, and unfortunately we do not yet have a way to address this question.

A first example of a ``bilinear Littlewood-Paley square function" was introduced by Lacey in \cite{bilinear-LP-Lacey}: if $\ds \hat{\Phi}$ is a smooth bump function supported on the interval $\left[ 0, 1\right]$, then
\begin{equation}
\label{def:Lacey-LP-square-F}
\left( f, g\right) \mapsto \left( \sum_{ k \in \rr{Z}} \lft \int_{\rr{R}^{2}} \hat{f}(\xi) \hat{g}(\eta) \hat{\Phi}( \xi -\eta - k) e^{2 \pi i x \left( \xi +\eta \right)} d \xi d \eta  \rg^2  \right)^{1/2}
\end{equation}
is a bounded operator from $L^p \times L^q$ into $L^2$, whenever $p, q \geq 2$ and $\ds \frac{1}{p}+\frac{1}{q}=\frac{1}{2}$. This work predates \cite{initial_BHT_paper}, where Lacey an Thiele prove the boundedness of the bilinear Hilbert transform, which is defined as:
\begin{equation}
\label{def:BHT}
BHT(f, g)(x)= \int_{\rr{R}^{2}} \hat{f}(\xi) \hat{g}(\eta) sgn (\xi -\eta) e^{2 \pi i x \left( \xi +\eta \right)} d \xi d \eta.
\end{equation}
The multiplier of the $BHT$ operator is singular along the line $\xi=\eta$, and for this reason its analysis is quite complicated. 

For the operator in \eqref{def:Lacey-LP-square-F}, the multipliers are given by $m_k(\xi, \eta):=\hat{\Phi}( \xi -\eta - k)$, are smooth, and are disjoint translations of the same multiplier $\hat{\Phi}( \xi -\eta)$. Later on, it was showed in \cite{smoothEDbilineairRF} and \cite{bilSqF-smooth-Fr} that this operator is bounded from $L^p \times L^q$ to $L^s$, for any $p, q >2$. The proof outside the local $L^2$ range (that is, for $s >2$) relies on the boundedness of the maximal truncation for the bilinear Hilbert transform, which is a rather deep result.

A non-smooth bilinear Littlewood-Paley square function for disjoint, arbitrary intervals was introduced in \cite{bilinLP}:
\begin{equation}
\label{def-LP-Frederic}
LP(f,g)(x)=\left( \sum_{k} \lft \int_{\rr{R}^{2}} \hat{f}(\xi) \hat{g}(\eta) \one_{\left[a_k, b_k \right]}( \xi -\eta ) e^{2 \pi i x \left( \xi +\eta \right)} d \xi d \eta  \rg^2  \right)^{1/2}.
\end{equation}
Here the family of multipliers is given by $m_k(\xi, \eta)=\one_{\left[a_k, b_k\right]}(\xi -\eta)$. So far, the only boundedness results that are known (inside the local $L^2$ range) correspond to intervals $\ds \lbrace \left[ a_k, b_k \right] \rbrace_k$ of equal lengths, and the proof hints at vector valued quantities for $BHT$. Even if the sharp cutoffs $\one_{\left[a_k, b_k \right]}$ are replaced by smooth functions $\hat{\Phi}_k$ which are adapted to the intervals $\left[a_k, b_k\right]$, the validity of a general result for arbitrary intervals is still unclear.

A sufficient condition for the boundedness of the square function $\left( \sum_k \vert T_k \vert^2 \right)^{1/2}$ inside the local $L^2$  range is a ``splitting " property of the operators $T_k$, in the sense that
\begin{equation}
\label{eq:mult-split}
T_k(f, g)=T_k\left( f \ast \check{\one}_{A_k}, g \ast \check{\one}_{B_k}  \right),
\end{equation}
where $\ds \lbrace A_k \rbrace_k$  and $\ds \lbrace B_k \rbrace_k$ are both collections of mutually disjoint intervals. This idea appears in \cite{bilinear_disc_multiplier}, \cite{DiestelGraf-maxBilOpDIlations} and \cite{DiestelRemarksBilLP}. In \cite{bilinear_disc_multiplier}, the authors are in fact expressing the bilinear disc multiplier as a sum of operators $T_k$, each of which satisfies \eqref{eq:mult-split}. In order to deduce the boundedness of $\sum_k T_k$ from the boundedness of the square function, one needs an extra orthogonality assumption: $\ds \langle T_k(f, g), h \rangle=\langle T_k \left( f \ast \check{\one}_{A_k},  g \ast \check{\one}_{B_k} \right),  h \ast \check{\one}_{C_k} \rangle$, where $\ds \lbrace A_k \rbrace_k, \lbrace B_k \rbrace_k$  and $\ds \lbrace C_k \rbrace_k$ are collections of mutually disjoint intervals. 

The operators in \eqref{def-LP-Frederic} and in \eqref{def:T_sharp} do not have such a splitting property and hence their analysis is much more complicated. Moreover, in both cases, the multipliers $m_k$ have infinite supports.

On the other hand, there are examples in \cite{vv_BHT} of operators satisfying $$T_k(f, g)=T\left( f \ast \check{\one}_{A_k}, g \ast \check{\one}_{B_k}  \right),$$ and hence \eqref{eq:mult-split}, and for which one can prove
\begin{equation}
\Big \|\left( \sum_k \lft T_k (f,g)\rg^r \right)^{1/r}   \Big \|_s \lesssim \|f\|_p \|g\|_q,
\end{equation}
for any $1 \leq r <\infty$, within a range larger that the local $L^2$ range. The operator $T$ can be for instance a paraproduct or the bilinear Hilbert transform. The proof relies on vector valued extensions for the operator $T$, and on a generalized version $RF_r$ of Rubio de Francia's square function.

We recall that the boundedness of $RF$, together with the Carleson-Hunt theorem (from \cite{initial_Carleson}, \cite{carleson-hunt}) imply through interpolation the boundedness of the operator
\begin{equation}
\label{def_RF_classic_r}
RF_r(f)(x)=\left( \sum_{ k } \lft \int \hat{f}(\xi) \one_{I_k}(\xi) e^{2 \pi i \xi x}    d \xi \rg^{r} \right)^{1/r}.
\end{equation}

\begin{theorem}[Rubio de Francia, \cite{RF}]
\label{thm_RF}
For any family of disjoint intervals, and any $r \geq 2$, $RF_r$ is a bounded operator from $L^p$ into $L^p$ whenever $p >r'$:
\begin{equation*}
\|RF_r(f)\|_p \lesssim \| f \|_p.
\end{equation*}
If $r=2$, $RF:L^p \to L^p$ for any $p \geq 2$.
\end{theorem}
The result is false for $r<2$, or for $p$ outside the range mentioned in Theorem \ref{thm_RF}. A counterexample can be constructed even for intervals of equal length. 

In dimensions $n \geq 2$, the only known result corresponds to $r=2$, and it will be interesting to understand if anything as generic as Theorem \ref{thm_RF} holds in higher dimensions.
~\\

Although we know how to perform a Fourier analysis associated with an arbitrary collection of intervals (or rectangles) in frequency for the linear setting, such a bilinear analogue was not sufficiently examined. Indeed, all the previously studied bilinear operators rely on a specific geometry (a line or a particular collections of lines). In the present paper, we study the following operator:
\begin{equation}
\label{eq:smooth_sqRF_r}
T_r\left( f, g \right)=\left( \sum_{\omega \in \Omega}\lft \int_{\rr{R}^2} \hat{f}(\xi) \hat{g}(\eta) \Phi_{\omega}(\xi, \eta) e^{2 \pi i x \left(\xi +\eta  \right)} d \xi d \eta   \rg^r \right)^{1/r},
\end{equation} 
where $\ds \lbrace  \omega \rbrace_{\omega \in \Omega}$ is an arbitrary collection of disjoint squares with sides parallel to the axes, and $\Phi_{\omega}$ are smooth bump functions adapted to $\omega$. We hope this will lead to a better understanding of the operator $LP$ from \eqref{def-LP-Frederic}, which is associated to an arbitrary collections of frequency strips.

We will prove the following result:
\begin{theorem}
\label{thm:main}
For any $r>2$, the operator $T_r$ maps $L^p \times L^q$ into $L^s$ boundedly, for any $r'< p, q \leq \infty, \frac{r'}{2}<s < r$, and $\ds \frac{1}{p}+\frac{1}{q}=\frac{1}{s}$.That is, 
\begin{equation}
\label{eq:thm:main}
\Big \| \left( \sum_{\omega \in \Omega} \lft \int_{\rr{R}^2} \hat{f}(\xi) \hat{g}(\eta) \Phi_{\omega}(\xi, \eta) e^{2 \pi i x \left(\xi +\eta  \right)} d \xi d \eta   \rg^r \right)^{1/r}  \Big \|_s \lesssim \big\|f\big\|_p \cdot \big\|g\big\|_q.
\end{equation}
\end{theorem}
\begin{figure}
    \centering
    \includegraphics[height=5cm]{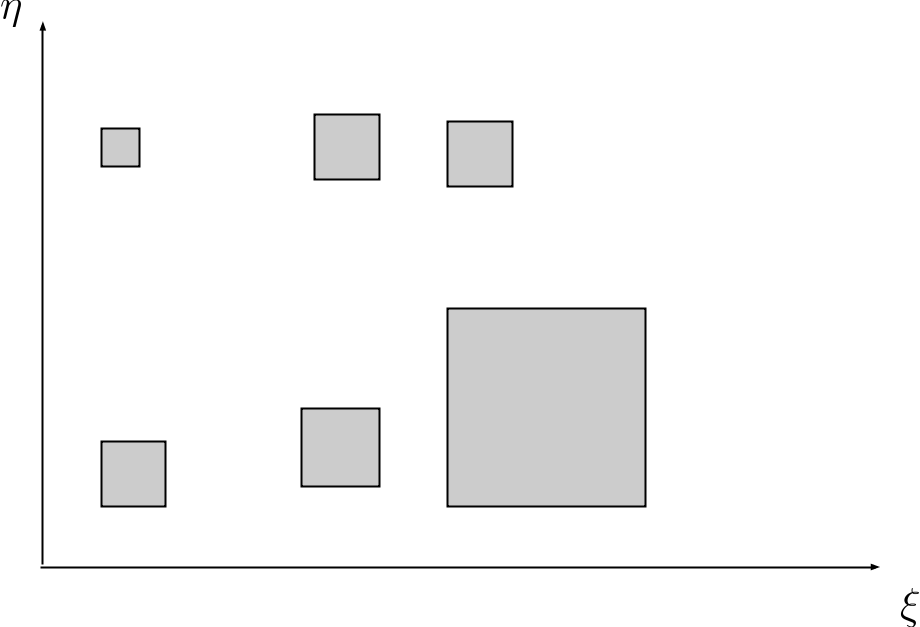}
    \caption{An arbitrary collection of squares}
    \label{fig:Omega}
\end{figure}
The arbitrary geometry on the frequency side, and hence of the time-frequency tiles, differentiates the operator $T_r$ from the classical operators from time-frequency analysis. The prominent examples of bilinear operators are associated to multipliers that are singular at a point (the classical Marcinkiewicz-Mikhlin-H\"{o}rmander multipliers from \cite{CoifMeyer-ondelettes}), along a line (the bilinear Hilbert transform \cite{initial_BHT_paper}), or more generally along curves (\cite{CamilPhDthesis} and \cite{bilinear_disc_multiplier}).

A few observations are in order:
\begin{enumerate}
\item[a)] If the projections of the squares onto the $\xi$ and respectively $\eta$ axes are mutually disjoint, then the boundedness of $T_r$ in the local $L^{r'}$ range is implied by an application of $RF_r$. This is similar to the principle in \eqref{eq:mult-split}.
\item[b)] We note that $s$ can be less than $1$, so the target space $L^s$ can be a quasi-Banach space.
\item[c)] If $r=\infty$, then $T_r : L^p \times L^q \to L^s$ for any $\ds 1< p, q < \infty$, and $\frac{1}{2}<s< \infty$. Here we only use the fast decay of the $\Phi_\omega$.
 As $r \to \infty$, we recover the expected range
\[
1<p, q <\infty, \enskip \frac{1}{2}< s< \infty.
\] 
\item[d)] The condition $r'<p, q$ appearing in Theorem \ref{thm:main} is necessary for the statement to be true in its generality. This becomes evident if one considers a particular configuration of squares of the same size, that are aligned along the strip $\ds 0 \leq \xi \leq 1$, or $\ds 0 \leq \eta \leq 1$.
\item[e)] There are no obvious $L^p$ estimates for the operator $T_r$, not even when $r=2$. This comes in contrast with the linear case , where $L^2$ estimates for $RF$ and its multi-dimensional generalizations from \eqref{def:RF-2d} are immediate.
\item[f)] Theorem \ref{thm:main} admits a multi-dimensional generalization, where $\Omega$ is an arbitrary collection of cubes in $\rr{R}^{2n}$. The proof is identical to the one-dimensional case.
\end{enumerate}

Up to now, it is not clear if $s<r$ is also a necessary condition, but it is an assumption that we need in our proof. Another requirement we cannot avoid is that $2<r$, leaving completely undecided the case of the square function, corresponding to $r=2$. A further question that remains open is whether the smooth cutoffs $\Phi_\omega$ can be replaced by non-smooth cutoffs: is 
\begin{equation}
\label{eq-sqRF_sharp}
T_r^{sharp}\left( f, g \right)(x):= \left( \sum_{\omega \in \Omega} \lft \int_{\rr{R}^2} \hat{f}(\xi) \hat{g}(\eta) \one_{\omega}(\xi, \eta) e^{2 \pi i x \left(\xi +\eta  \right)} d \xi d \eta   \rg^r \right)^{1/r},
\end{equation} 
a bounded operator from $L^p \times L^q$ into $L^s$? The only ``easy" case is $r=\infty$, for which the operator is bounded from $\ds L^p \times L^q \to L^s$ for any $1<p, q< \infty$, and $\dfrac{1}{2}< s< \infty$. In spite of the similarity with the smooth operator $T_\infty$, and in spite of being bounded within the same range, the non-smooth case exhibits additional difficulties: in order to prove the boundedness of $T_\infty^{sharp}$, one needs to invoke the Carleson-Hunt theorem.
~\\

In the proof of Theorem \ref{thm:main}, we will be using Banach-valued restricted weak type interpolation, similar to the presentations in \cite{wave_packet}, \cite{vv_BHT}. The Banach space associated to our operator $T_r$ is $\ell^r$ indexed by the collection $\Omega$ of disjoint squares. Its dual is the space $\ell^{r'}$ indexed also by $\Omega$. Theorem \ref{thm:main} reduces to proving restricted weak type estimates for the trilinear form $\Lambda$ associated to a model operator for $T_r$:
\begin{proposition}
\label{prop:main}
Let $F, G$ and $H$ be measurable subsets of $\rr{R}$, of finite measure, with $\vert H \vert=1$. Then one can construct a major subset $H' \subseteq H$, $\ds \vert H'\vert > \vert H \vert/2$, so that
\begin{equation}
\label{eq:main:dual_form}
\big \vert \Lambda \left(f, g, h \right) \big \vert \lesssim \vert F \vert^{\frac{1}{p}} \vert G \vert^{\frac{1}{q}} \vert H \vert^{\frac{1}{s'}},
\end{equation} 
whenever the functions $f, g, h =\lbrace h_\omega\rbrace_{\omega \in \Omega}$ satisfy 
\begin{equation}
 \label{conditions_F,G,H} \ds \vert f\vert \leq \one_F, \quad \vert g\vert \leq \one_G, \quad  \left( \sum_{\omega} \vert h_\omega  \vert^{r'}   \right)^{1/{r'}} \leq \one_{H'},
\end{equation}
and the exponents $p, q, s$ satisfy $r'< p, q < \infty, \frac{r'}{2}< s < r$, and $\ds \frac{1}{p}+\frac{1}{q}=\frac{1}{s}$.
\end{proposition}

The paper is organized in the following way: in Section \ref{sec:model_op} we describe the discretization of the operator $T_r$, and introduce the new \emph{column} and \emph{row} structures of tiles. Related to these notions, we define new sizes and energies in Section \ref{sec:size_and_en}, which will be used in Section \ref{sec:gen-est} in order to establish a generic estimate for the trilinear form. Some refinements of the energy estimates are performed in Section \ref{sec-localization}, and the proof of Theorem \ref{thm:main} is presented in Section \ref{sec:main-proof}. Finally, in Section \ref{sec:an-application}, we present an application to generalized Bochner-Riesz bilinear multiplier for rough domains.

\subsection*{Acknowledgements}Both authors are supported by ERC project FAnFArE no. $637510$.

\section{The Model Operator and the Organization of the time-frequency Tiles}
\label{sec:model_op}

We start with a few definitions:
\begin{definition}
A \emph{time-frequency tile} is a rectangle $P= I \times \omega$ of area $1$, where $I$ and $\omega$ are dyadic intervals.

A tri-tile is a tuple $\ds s =\left( I_s \times\omega_1, I_s \times\omega_2, I_s \times\omega_3 \right)$, where each $\ds s_i=I_s \times \omega_i$ is a tile.
\end{definition}
\begin{definition}
For a fixed interval $I$, we denote
\begin{equation}
\ci_I(x):=\left( 1+\frac{\dist(x, I)}{\vert I \vert}\right)^{-10}.
\end{equation}
We say that a function $\phi$ is \emph{adapted} to $I$ if 
\[
\vert\phi^{\left( k \right)}(x)  \vert \leq C_{k, M} \frac{1}{\vert I \vert^k} \ci_I^M(x),
\]
for sufficiently many derivatives, and $M >0$ a large number.

Given a tile $P=I_P \times \omega_P$, we say that $\phi$ is a \emph{wave packet} associated to $P$ if $\phi$ is adapted to $I_P$, $\hat{\phi}$ is adapted to $\omega_P$, and $\hat{\phi}$ is supported inside $\dfrac{11}{10} \omega_P$.
\end{definition}

A first simplification of the operator $T_r$ consists in assuming that the squares $\omega \in \Omega$ are dyadic. This reduction is possible because the smooth cutoff $\Phi_\omega$, supported on $\omega=\omega_1 \times \omega_2$ can be replaced by a smooth cutoff supported on $\tilde{\omega}_1 \times \tilde{\omega}_2$, where $\ds \vert \omega_i \vert \sim \vert \tilde{\omega}_i \vert$ and the intervals $\tilde{\omega}_i$ are either dyadic intervals or shifted dyadic intervals (they are shifted a third of a unit to the left or to the right ). The function $\hat{\Phi}_\omega$ is replaced by its double Fourier series on $\ds \tilde{\omega}_1 \times \tilde{\omega}_2$:
\begin{equation*}
\hat{\Phi}_\omega(\xi, \eta)=\sum_{l, k}c_{l, k} \hat{\phi}_{\omega, 1, l, k}(\xi) \hat{\phi}_{\omega, 2, l, k}(\eta),
\end{equation*}
where $\hat{\phi}_{\omega, i, l, k}$ is smooth, supported on $ \dfrac{11}{10} \tilde{\omega}_i$, $\equiv 1$ on $\tilde{\omega}_i$. Since we will be working with the trilinear form associated to the operator $T_r$, we write 
\begin{equation*}
\hat{\Phi}_\omega(\xi, \eta)=\sum_{l, k} c_{l, k}\hat{\phi}_{\omega, 1, l, k}(\xi) \hat{\phi}_{\omega, 2, l, k}(\eta)= \sum_{l, k}c_{l, k} \hat{\phi}_{\omega, 1, l, k}(\xi) \hat{\phi}_{\omega, 2, l, k}(\eta) \hat{\phi}_{\omega, 3, l, k}(\xi + \eta) ,
\end{equation*}
where $\hat{\phi}_{\omega, 3, l, k}(\eta)$ is smooth, supported on $ \dfrac{11}{10} \tilde{\omega}_3$, $\equiv 1$ on $\tilde{\omega}_3$. Here $\tilde{\omega}_3$ is a (shifted) dyadic interval containing $\tilde{\omega}_1+\tilde{\omega}_2$, and so that $\ds \vert \tilde{\omega}_3\vert \sim \tilde{\omega}_1+\tilde{\omega}_2$. 

The fast decay of the Fourier coefficients (implied by the smoothness of $\hat{\Phi}_\omega$) ensures that the boundedness of the general case can be deduced from the boundedness of the dyadic case. Working with dyadic intervals simplifies the time-frequency analysis of the operator, merely because any two dyadic intervals are either disjoint or one of them is contained inside the other one.

In this way, we obtain a model operator of $T_r$ associated to a finite collection $\rr{S}$ of  tri-tiles  of the form
\begin{equation}
\label{def:tile}
s=\left(I_s \times \omega_1, I_s \times \omega_2, I_s \times \omega_3 \right).
\end{equation}
Here $\omega=\omega_1\times \omega_2 \in \Omega$ is a square contained in the collection $\Omega$ , and $\omega_3 \sim \omega_1+\omega_2$. In this case, if $s$ is of the form $\eqref{def:tile}$, we use the notation $\omega_s=\omega$, and $\omega_{s_j}=\omega_j$, for $1 \leq j \leq 3$. For any subcollection $\rr{S}'$ of tiles, we define
\[
\Omega(\rr{S}'):=\lbrace \omega \in \Omega: \exists s \in \rr{S}' \text{ such that  } \omega=\omega_1 \times \omega_2    \rbrace.
\]
Note that a frequency square $\omega$ could correspond to several tri-tiles: given $\omega \in \Omega(\rr{S'})$, there are possibly several tiles $\ds s, s' \in \rr{S}$ so that $\omega_s=\omega_{s'}$. 

Then the model operator for $T_r$ is given by 
\begin{equation}
(f, g) \mapsto  \left( \sum_{\omega \in \Omega} \lft \sum_{\substack{s \in \rr{S} \\ \omega_s=\omega}} \vert I_s \vert^{-\frac{1}{2}}\langle f, \phi_{s_1}\rangle \langle g, \phi_{s_2}\rangle \phi_{s_3}(x) \rg^r \right)^{1/r},
\end{equation}
where the functions $\phi_{s_j}$ are wave packets associated to the tiles $s \in \rr{S}$. The trilinear form, obtained by dualization with a function $\ds h =\lbrace h_\omega \rbrace_{\omega \in \Omega}$, is given by
\begin{equation}
\label{def:trilinean_form}
\Lambda_{\rr{S}} \left(f, g, h \right):=\sum_{s \in \rr{S}} \vert I_s \vert^{-1/2} \langle f, \phi_{s_1} \rangle \langle g, \phi_{s_2} \rangle \langle h_s, \phi_{s_3} \rangle,
\end{equation}
where $h_s=h_\omega$ whenever $\omega_s=\omega$.

\subsection{Columns and Column estimate}~\\
For the model operator of $T_r$, the geometry of the tiles is unconventional, and the \emph{tree}-structures from \cite{initial_BHT_paper} or \cite{multilinearMTT}, are replaced here by \emph{columns} and \emph{rows}. In this situation, there is no relation between the length of a tile in the column and the distance to the ``top" frequency. We have the following definitions:
\begin{definition}
A \emph{column} with top $t$ is a subcollection $\mathcal{C} \subseteq \rr{S}$ with the property that for all $s \in \mathcal{C}$,
\[
I_s \subseteq I_t \quad \text{and     } \omega_{t_1} \subseteq \omega_{s_1}.
\]
We denote the top tile of the column $\cc$ as $t_{\cc}:=I_{\cc} \times \omega_{\cc}$. Since the tiles are overlapping in the $\xi$ direction, they are going to be disjoint in the $\eta$ direction: \textit{ for all $s \in \mathcal{C}$, the intervals $\omega_{s_2}$ are mutually disjoint.}

Similarly, a \emph{row} with top $t$ is a subcollection $\ii{R} \subseteq \rr{S}$ with the property that for all $s \in \mathcal{R}$,
\[
I_s \subseteq I_t \quad \text{and     } \omega_{t_2} \subseteq \omega_{s_2}.
\]
We denote the top as $t_{\ii{R}}:=I_{\ii{R}} \times \omega_{\ii{R}}$. This time, the intervals $\ds \lbrace \omega_{s_1}\rbrace_{s \in \ii{R}}$ are mutually disjoint.
\end{definition}

\begin{definition}
\label{def:mutually-disj}
We say that the columns $\cc_1, \ldots, \cc_N$ are \emph{mutually disjoint} if they are disjoint sets of tri-tiles (that is, $\cc_i \cap \cc_j =\emptyset$  for all $i \neq j$), and
\[
\lbrace  I_{\cc_j} \times \omega_{\cc_j, 1}  \rbrace_{1 \leq j \leq N}
\]
represents a collection of mutually disjoint tiles: $\ds  I_{\cc_i} \times \omega_{\cc_i, 1} \cap  I_{\cc_j} \times \omega_{\cc_j, 1} =\emptyset$ for all $i \neq j$.
\emph{Mutually disjoint rows} are defined in a similar manner, but this time 
\[
\lbrace I_{\ii{R}_j} \times \omega_{\ii{R}_j, 2} \rbrace_{1 \leq j \leq N}
\]
form a collection of mutually disjoint tiles.
\end{definition}

The columns and rows are configurations suitable for the time-frequency analysis of $\Lambda_{\rr{S}}$: if we restrict our attention to columns, we get a nice estimate in Proposition \ref{prop:column est}, and similarly for rows. These estimates give rise to new ``sizes", which will be introduced in Section \ref{sec:size_and_en}.

\begin{proposition}
\label{prop:column est}
Let $\mathcal{C}$ be a column with top $t$. Then we have the following estimate:
\begin{align}
\label{eq: col est}
\big \vert \Lambda_{\mathcal{C}} (f, g, h) \big \vert & \lesssim \sup_{s \in \cc} \frac{\vert \langle f, \phi_{s_1}  \rangle \vert}{|I_s|^{1/2}} \cdot \left( \sup_{s \in \cc} \frac{\vert \langle  g, \phi_{s_2}  \rangle\vert}{|I_s|^{\frac{1}{2}}} \right)^{\frac{r-2}{r}} \\
&\cdot \left( \frac{1}{|I_t|}\sum_{s \in \cc} \vert \langle g, \phi_{s_2}  \rangle  \vert^2 \right)^{1/r} \cdot \left( \frac{1}{|I_t|}\int_{\rr{R}}  \sum_{\omega \in \Omega(\cc)} \lft \mathcal{M}\left( h_\omega \cdot \ci_{I_t}^{M} \right)\rg^{r'} \cdot \one_{I_t} dx \right)^{1/{r'}} \cdot |I_t|. \nonumber
\end{align}
Here $M>0$ can be as large as we wish and the implicit constant will depend on $M$.
\begin{proof}
First, note that we have (following H\"older's inequality), for every $\alpha >0$,
\begin{align*}
 \lft\Lambda_{{\mathcal C}}(f,g,h)\rg &= \lft \sum_{s \in \mathcal{C}}  \vert I_s  \vert^{-1/2} \langle f, \phi_{s_1}\rangle \langle g, \phi_{s_2}  \rangle \langle h_s, \phi_{s_3}  \rangle\rg \\
&\lesssim \sup_{s \in \cc} \frac{\vert \langle f, \phi_{s_1}  \rangle \vert}{|I_s|^{1/2}} \cdot \left( \sum_{s \in \cc} \left(\frac{|I_s|}{|I_t|}  \right)^{- \alpha r}  \vert \langle g, \phi_{s_2}  \rangle \vert^r  \right)^{1/r}  \cdot  \left( \sum_{s \in \cc} \left(\frac{|I_s|}{|I_t|}  \right)^{ \alpha r'}  \vert \langle h_s, \phi_{s_3}  \rangle  \vert^{r'} \right)^{1/{r'}},
\end{align*}
and in what follows we will focus on the second and third term. For $g$, since $r>2$, we have
\begin{align*}
&\left( \sum_{s \in \cc} \left(\frac{|I_s|}{|I_t|}  \right)^{- \alpha r}  \vert \langle g, \phi_{s_2}  \rangle  \vert^r \right)^{1/r} \\
&\qquad =\left(\sum_{s \in \cc} \frac{\vert \langle  g, \phi_{s_2}  \rangle\vert^{r-2}}{|I_s|^{\frac{r-2}{2}}}   \cdot \vert I_s \vert^{\frac{r-2}{2}} \cdot  \left(\frac{|I_s|}{|I_t|}  \right)^{- \alpha r}   \vert \langle g, \phi_{s_2}  \rangle  \vert^2 \right)^{1/r}  \\
&\qquad \lesssim  \left( \sup_{s \in \cc} \frac{\vert \langle  g, \phi_{s_2}  \rangle\vert}{|I_s|^{\frac{1}{2}}} \right)^{\frac{r-2}{r}} \cdot \left(\sum_{s \in \cc} \vert \langle g, \phi_{s_2}  \rangle  \vert^2 \right)^{1/r} \cdot \vert I_t \vert^\alpha,
\end{align*}
provided $\ds \alpha r =\frac{r-2}{2}$, which is equivalent to $\ds \alpha =\frac{1}{2}-\frac{1}{r} >0$. The last term will be slightly more technical:
\begin{align*}
& \sum_{s \in \cc} \left(\frac{|I_s|}{|I_t|}  \right)^{ \alpha r'}  \vert \langle h_s, \phi_{s_3}  \rangle  \vert^{r'} \\
&\qquad = \sum_{l \geq 0} \ \sum_{|\omega|^{-1}=2^{-l}|I_t|} \ \sum_{\substack{s \in \cc\\ \omega_s=\omega}} \left(\frac{|I_s|}{|I_t|}  \right)^{ \alpha r'}  \frac{\vert \langle h_\omega, \phi_{s_3} \rangle \vert^{r'}}{|I_s|^{r'/2}} \cdot \frac{|I_s|^{r'/2}}{|I_t|^{r'/2}} \cdot |I_t|^{r'/2}\\
&\qquad =|I_t|^{r'/2} \sum_{l \geq 0} 2^{-l r' \left( \alpha +\frac{1}{2}  \right)} \sum_{|\omega|^{-1}=2^{-l}|I_t|} \ \sum_{\substack{s \in \cc\\ \omega_s=\omega}} \left( \frac{\vert \langle h_\omega, \phi_{s_3} \rangle \vert}{|I_s|^{1/2}}  \right)^{r'}.
\end{align*}
Now we observe that
\begin{equation}
\label{eq:how_we_got_M}
\frac{\vert \langle h_\omega, \phi_{s_3} \rangle \vert}{|I_s|^{1/2}} \lesssim \inf_{y \in I_s} \lft \mathcal{M} \left( h_{\omega} \cdot \ci_{I_s}^{M} \right) (y)\rg \lesssim \inf_{y \in I_s} \lft \mathcal{M} \left( h_{\omega} \cdot \ci_{I_t}^{M} \right) (y)\rg,
\end{equation}
since the bump functions $\phi_{s_3}$ are $L^2$-normalized and adapted to $I_s$. This implies that 
\begin{align*}
& \sum_{s \in \cc} \left(\frac{|I_s|}{|I_t|}  \right)^{ \alpha r'}  \vert \langle h_s, \phi_{s_3}  \rangle  \vert^{r'} \\
&\lesssim \vert  I_t \vert^{\frac{r'}{2}} \sum_{l \geq 0} 2^{-l r' \left(\alpha+\frac{1}{2}\right)} \sum_{\vert \omega\vert^{-1}=2^{-l} \vert I_t\vert} \ \sum_{\substack{s \in \cc\\ \omega_s=\omega}} \frac{1}{\vert I_s \vert} \int_{\rr{R}} \lft \mathcal{M} \left( h_{\omega} \cdot \ci_{I_t}^{M} \right) (y)\rg^{r'} \cdot \one_{I_s} dx \\
&\lesssim \vert  I_t \vert^{\frac{r'}{2}}\sum_{l \geq 0} 2^{-l r' \left(\alpha+\frac{1}{2} -\frac{1}{r'}\right)} \sum_{\vert \omega\vert^{-1}=2^{-l} \vert I_t\vert} \ \sum_{\substack{s \in \cc\\ \omega_s=\omega}} \frac{1}{\vert I_t \vert} \int_{\rr{R}} \lft \mathcal{M} \left( h_{\omega} \cdot \ci_{I_t}^{M} \right) (y)\rg^{r'} \cdot \one_{I_s} dx \\
&\lesssim \vert  I_t \vert^{\frac{r'}{2}} \cdot \frac{1}{\vert I_t \vert} \int_{\rr{R}} \sum_{\omega \in \Omega\left( \cc \right)} \lft \mathcal{M} \left( h_{\omega} \cdot \ci_{I_t}^{M} \right) (y)\rg^{r'} \cdot \one_{I_t} dx.
\end{align*} 
Above, the definition of $\alpha$ yields  $\alpha+\frac{1}{2}=\frac{1}{r'}$.

Carefully adding all these estimates together, we get that  
\begin{align*}
&\big \vert \sum_{s \in \mathcal{C}}  \vert I_s  \vert^{-1/2} \langle f, \phi_{s_1}\rangle \langle g, \phi_{s_2}  \rangle \langle h_s, \phi_{s_3}  \rangle\big \vert \\
&\lesssim \sup_{s \in \cc} \frac{\vert \langle f, \phi_{s_1}  \rangle \vert}{|I_s|^{1/2}} \cdot \left( \sup_{s \in \cc} \frac{\vert \langle  g, \phi_{s_2}  \rangle\vert}{|I_s|^{\frac{1}{2}}} \right)^{\frac{r-2}{r}} \\
&\cdot \left(\sum_{s \in \cc} \vert \langle g, \phi_{s_2}  \rangle  \vert^2 \right)^{1/r} \cdot \vert I_t \vert^{1/2- 1/r} \cdot  |I_t|^{1/2-1/{r'}} \left( \int_{\rr{R}} \sum_{\omega \in \Omega\left( \cc \right)} \lft \mathcal{M} \left( h_{\omega} \cdot \ci_{I_t}^{M} \right) (y)\rg^{r'} \cdot \one_{I_t} dx \right)^{1/{r'}},
\end{align*}
which is precisely \eqref{eq: col est}.
\end{proof}
\end{proposition}

Similarly, we have estimates for a row $\ii{R}$:
\begin{proposition}
\label{prop:row est}
If $\ii{R} \subseteq \rr{S}$ is a row of top $t$, then,
\begin{align}
\label{eq: row est}
\big \vert \Lambda_{\ii{R}} (f, g, h) \big \vert & \lesssim \left( \sup_{s \in \ii{R}} \frac{\vert \langle  f, \phi_{s_1}  \rangle\vert}{|I_s|^{\frac{1}{2}}} \right)^{\frac{r-2}{r}} \cdot \left( \frac{1}{|I_t|}\sum_{s \in \ii{R}} \vert \langle f, \phi_{s_1}  \rangle  \vert^2 \right)^{1/r} \\
&\cdot \sup_{s \in \ii{R}} \frac{\vert \langle g, \phi_{s_2}  \rangle \vert}{|I_s|^{1/2}} \cdot \left( \frac{1}{|I_t|}\int_{\rr{R}}  \sum_{\omega \in \Omega(\ii{R})} \lft \mathcal{M}\left( h_\omega \cdot \ci_{I_t}^{M} \right)\rg^{r'} \cdot \one_{I_t} dx \right)^{1/{r'}} \cdot |I_t|. \nonumber
\end{align}
\end{proposition}

\begin{proposition}
If $\cc \subseteq \rr{S}$ is a column, then 
\begin{equation}
 \frac{1}{|I_{\cc}|}\sum_{s \in \cc} \vert \langle g, \phi_{s_2}  \rangle  \vert^2  \lesssim   \frac{1}{\vert I_{\cc} \vert} \int_{\rr{R}} \vert g(x) \vert^2 \cdot \ci_{I_{\cc}}^{10} dx.
\end{equation}
\begin{proof}
This follows easily from orthogonality arguments, and the fast decay of the bump functions.
\end{proof}
\end{proposition}

\section{Sizes and Energies}
\label{sec:size_and_en}
Motivated by the estimates in Propositions \ref{prop:column est} and \ref{prop:row est}, we define \emph{sizes} with respect to a collection $\rr{S}$ of tiles,  in the following way:
\begin{definition} For $f\in L^1_{\textrm{loc}}({\mathbb R})$ and $\rr{S}$ a collection of tiles, we set
\begin{equation}
\label{eq:def:size_f} \ssize_{\rr{S}} \left( f \right):= \sup_{s \in \rr{S}} \frac{\vert \langle f, \phi_{s_1} \rangle \vert}{\vert I_s \vert^{1/2}}.
\end{equation}
Similarly for $g\in L^1_{\textrm{loc}}({\mathbb R})$, 
\begin{equation}
\label{eq:def:size_g} \ssize_{\rr{S}} \left( g \right):= \sup_{s \in \rr{S}} \frac{\vert \langle g, \phi_{s_2} \rangle \vert}{\vert I_s \vert^{1/2}}.
\end{equation}
\end{definition} 

\begin{definition}
For a sequence $\ds h=\lbrace h_\omega\rbrace_{\omega \in \Omega}$ of $L^1_{\textrm{loc}}(\ell^{r'}(\Omega))$, the size is defined as
\begin{equation}
\label{eq:def:size_h}
\ssize_{\rr{S}} \left( h \right):= \sup_{\substack{\mathcal{T} \subseteq \rr{S} \\ \mathcal{T} \text{ column or row} \\ \text{ with top  }t}} \left( \frac{1}{|I_t|}\int_{\rr{R}}  \sum_{\omega \in \Omega(\mathcal{T})} \lft \mathcal{M}\left( h_\omega \cdot \ci_{I_t}^{M} \right)\rg^{r'} \cdot \one_{I_t} dx \right)^{1/{r'}}.
\end{equation}
\end{definition}

Correspondingly, the \emph{energies} with respect to a collection $\rr{S}$ are constructed as follows:
\begin{definition}\label{def:en_f}
 For $f\in L^1_{\textrm{loc}}({\mathbb R})$, we define
\begin{equation}
\label{eq:def:en_f}
\eenergy_{\rr{S}}\left( f \right):= \sup_{n \in \rr{Z}} \ 2^n \left( \sum_{ \cc \in \mathfrak{C}} \vert I_{\cc}\vert \right)^{1/2},
\end{equation}
where $\mathfrak{C}$ ranges over all collections of mutually disjoint columns $\cc \subseteq \rr{S}$ (see Definition \ref{def:mutually-disj}), so that 
\[
\frac{\vert \langle f, \phi_{s_1} \rangle \vert}{\vert I_s \vert^{1/2}} \leq 2^{n+1}, \enskip \text{for all } s \in \cc
\]
and whose tops satisfy
\[
\frac{\vert \langle f, \phi_{t_{\cc,1}} \rangle \vert}{\vert I_{\cc}\vert^{1/2}} \geq 2^n, \enskip \text{for all } \cc \in \mathfrak{C}.
\]
Also, we have for $g\in L^1_{\textrm{loc}}({\mathbb R})$
\begin{equation}
\label{eq:def:en_g}
\eenergy_{\rr{S}}\left( g \right):= \sup_{n \in \rr{Z}} \ 2^n \left( \sum_{ \ii{R} \in \mathfrak{R}} \vert I_{\ii{R}}\vert \right)^{1/2},
\end{equation}
where $\mathfrak{R}$ ranges over all collections of mutually disjoint rows $\ii{R} \subseteq \rr{S}$ with the property that
\[
\frac{\vert \langle g, \phi_{s_2} \rangle \vert}{\vert I_s \vert^{1/2}} \leq 2^{n+1}, \enskip \text{for all } s \in \ii{R},
\]
and whose tops satisfy
\[
\frac{\vert \langle g, \phi_{t_{\ii{R},2}} \rangle \vert}{\vert I_{\ii{R}}\vert^{1/2}} \geq 2^n, \enskip \text{for all } \ii{R} \in \mathfrak{R}.
\]
\end{definition}
\begin{definition}
\label{def:en_h}
Given a sequence of functions $\ds h=\lbrace h_\omega \rbrace_{\omega \in \Omega}$, and a collection of tiles $\rr{S}$, we set
\begin{equation}
\label{eq:def:en_h}
\eenergy_{\rr{S}}\left( h \right):=\sup_{n \in \rr{Z}} \ 2^n \left( \sup_{\mathcal{T} \in \mathfrak{T}} \vert I_{\mathcal{T}} \vert\right)^{1/{r'}},
\end{equation}
where $\mathfrak{T}$ ranges over all collections of mutually disjoint rows and mutually disjoint columns (with top $t=I_{\mathcal{T}} \times \omega_{\mathcal{T}}$) satisfying
\begin{equation}
\label{eq:size_h:constr}
\left( \frac{1}{|I_{\mathcal{T}}|}\int_{\rr{R}}  \sum_{\omega \in \Omega(\mathcal{T})} \lft \mathcal{M}\left( h_\omega \cdot \ci_{I_{\mathcal{T}}}^{M} \right)\rg^{r'} \cdot \one_{I_{\mathcal{T}}} dx \right)^{1/{r'}} \geq 2^n.
\end{equation}
In fact, $\mathfrak{T}$ will be the union of a collection $\mathfrak{C}$ of mutually disjoint columns and a collection $\mathfrak{R}$ of mutually disjoint rows, where every column or row satisfies \eqref{eq:def:en_h}.
\end{definition}

We will need to bound these quantities, but this procedure is rather standard.
\begin{proposition}
\label{prop:size_est}
For any locally integrable function $f$ and any collection of tiles $\rr{S}$
\begin{equation}
\label{eq:size_est}
\ssize_{\rr{S}}\left( f \right) \lesssim \sup_{s \in \rr{S}} \ \frac{1}{\vert I_s \vert} \int_{\rr{R}} \vert f(x)\vert \cdot \ci_{I_s}^M(x) dx, 
\end{equation}
for $M > 0$ arbitrarily large, with the implicit constant depending on $M$. A similar estimate holds for $\ssize_{\rr{S}}\left( g \right)$.
\end{proposition}

In Proposition \ref{prop:size_est}, we only make use of the fast decay of the wave packets $\phi_{s_j}$. However, for the energy, it is of utmost importance that the top tiles $\ds \lbrace I_{\cc} \times \omega_{\cc, 1} \rbrace_{\cc \in \mathfrak{C}}$ are mutually disjoint tiles, whenever $\mathfrak{C}$ represents a collection of mutually disjoint columns.

\begin{proposition}
\label{propo:en_est_f}
For every functions $f,g\in L^2({\mathbb R})$ we have
\[
\eenergy_{\rr{S}}\left( f \right) \lesssim \|  f \|_2 \qquad \textrm{and} \qquad  \eenergy_{\rr{S}}\left( g \right) \lesssim \|  g \|_2.
\]
\begin{proof}
This is very similar to Lemma 5.1 from \cite{wave_packet}, but we present the details for completeness. Assume that $n$ and $\mathfrak{C}$ are energy maximizers in Definition \ref{def:en_f}. For the top intervals we have
\[
2^{n} \leq \frac{\left| \langle f, \phi_{t_{\cc, 1}}  \rangle\right|}{\vert I_{\cc} \vert^{1/2}} \leq 2^{n+1}, \enskip \text{for all  }\enskip \cc \in \mathfrak{C}.
\]
Following the definition, we have 
\[
\left( \eenergy_{\rr{S}}(f) \right)^2= 2^{2n} \sum_{\cc \in \mathfrak{C}} \vert I_{\cc} \vert \leq \sum_{\cc \in \mathfrak{C}} \left| \langle f, \phi_{t_{\cc, 1}} \rangle \right|^2 \leq \|f\|_2 \cdot \left\| \sum_{\cc \in \mathfrak{C}} \langle f, \phi_{t_{\cc, 1}}  \rangle \phi_{t_{\cc, 1}}  \right\|_2,
\]
and it will be enough to prove $\ds \left\| \sum_{\cc \in \mathfrak{C}} \langle f, \phi_{t_{\cc, 1}}  \rangle \phi_{t_{\cc, 1}}  \right\|_2 \leq \left( 2^{2n} \sum_{\cc \in \mathfrak{C}} \vert I_{\cc} \vert \right)^{1/2}$. To this end, we compute:
\begin{align*}
\left\| \sum_{\cc \in \mathfrak{C}} \langle f, \phi_{t_{\cc, 1}}  \rangle \phi_{t_{\cc, 1}}  \right\|_2^2=\sum_{\cc, \cc' \in \mathfrak{C}} \langle f, \phi_{t_{\cc, 1}} \rangle \langle f, \phi_{t_{\cc', 1}}\rangle \langle \phi_{t_{\cc, 1}}, \phi_{t_{\cc', 1}}\rangle.
\end{align*}

The only way $\ds \langle \phi_{t_{\cc, 1}}, \phi_{t_{\cc', 1}}\rangle \neq 0$ is if $\ds \omega_{\cc, 1} \cap \omega_{\cc', 1} \neq \emptyset$. By symmetry, we can estimate
\begin{align*}
&\left\| \sum_{\cc \in \mathfrak{C}} \langle f, \phi_{t_{\cc, 1}}  \rangle \phi_{t_{\cc, 1}}  \right\|_2^2 \lesssim \sum_{\substack{\cc, \cc' \in \mathfrak{C}\\ \omega_{\cc, 1} \subseteq \omega_{\cc', 1}}} \left| \langle f, \phi_{t_{\cc, 1}} \rangle \langle f, \phi_{t_{\cc', 1}}\rangle \langle \phi_{t_{\cc, 1}}, \phi_{t_{\cc', 1}}\rangle \right|\\
&\lesssim \sum_{\cc \in \mathfrak{C}} \sum_{\substack{\cc' \in \mathfrak{C} \\ \omega_{\cc, 1} \subseteq \omega_{\cc', 1}}} \left| \langle f, \phi_{t_{\cc, 1}} \rangle \right| \cdot \left| \langle f, \phi_{t_{\cc', 1}} \rangle \right| \cdot \left| \langle \phi_{t_{\cc, 1}}, \phi_{t_{\cc', 1}} \rangle\right| \\
&\lesssim \sum_{\cc \in \mathfrak{C}} \sum_{\substack{\cc' \in \mathfrak{C} \\ \omega_{\cc, 1} \subseteq \omega_{\cc', 1}}} 2^{n+1} \vert I_{\cc} \vert^{1/2} \cdot 2^{n+1} \vert I_{\cc'} \vert^{1/2} \left| \langle \phi_{t_{\cc, 1}}, \phi_{t_{\cc', 1}} \rangle \right|.
\end{align*}
The last inequality is a consequence of the energy definition, since any tri-tile in $\mathfrak{C}$ has the property that $\ds \left| \langle f, \phi_{s, 1} \rangle \right| \leq 2^{n+1} \vert I_s \vert^{1/2}$.
We employ again the fast decay of the wave packets: since $\omega_{\cc, 1} \subseteq \omega_{\cc', 1}$, we have $\vert I_{\cc'} \vert \leq \vert I_{\cc} \vert$ and 
\[
\left| \langle \phi_{t_{\cc, 1}}, \phi_{t_{\cc', 1}} \rangle \right| \lesssim \left(\frac{\vert I_{\cc'}\vert}{\vert I_{\cc} \vert}\right)^{1/2} \left(  1 +\frac{\dist(I_{\cc}, I_{\cc'})}{\vert I_c \vert} \right)^{-100}.
\]
Hence, we have
\begin{align*}
&\left\| \sum_{\cc \in \mathfrak{C}} \langle f, \phi_{t_{\cc, 1}}  \rangle \phi_{t_{\cc, 1}}  \right\|_2^2  \lesssim \sum_{\cc \in \mathfrak{C}} 2^{2n} \sum_{\substack{\cc' \in \mathfrak{C} \\ \omega_{\cc, 1} \subseteq \omega_{\cc', 1}}} \int_{I_{\cc'}} \left(1 +\frac{\dist( x, I_{\cc})}{\vert I_c \vert}  \right)^{-100} dx \\
&\lesssim \sum_{\cc \in \mathfrak{C}} 2^{2n} \int_{\rr{R}} \left(1 +\frac{\dist( x, I_{\cc})}{\vert I_c \vert}  \right)^{-100} dx  \lesssim \sum_{\cc \in \mathfrak{C}} 2^{2n} \vert I_{\cc} \vert.
\end{align*}
Whenever we have a subcollection $\ds \lbrace \cc' \in \mathfrak{C}:  \omega_{\cc, 1} \subseteq \omega_{\cc', 1}\rbrace$, the spatial intervals $ I_{\cc'}$ are mutually disjoint. This is implied by the pairwise disjointness of the tiles $\ds \lbrace I_{\cc} \times \omega_{\cc,1} \rbrace$. The last inequality completes the energy estimate.

We note that the disjointness of the tiles $\ds \lbrace I_{\cc} \times \omega_{\cc, 1} \rbrace_{\cc \in \mathfrak{C}}$ is not sufficient for concluding 
\begin{equation}
\label{eq:bessel}
\sum_{\cc \in \mathfrak{C}} \left| \langle f, \phi_{t_{\cc, 1}} \rangle \right|^2 \lesssim \|f\|_2^2.
\end{equation}
In fact, a counterexample is presented in \cite{wave_packet}. However, besides the mutually disjointness of the tiles in the above collection, we also use the condition on the tops of the columns:
$$2^{n} \leq \frac{\left| \langle f, \phi_{t_{\cc, 1}}  \rangle\right|}{\vert I_{\cc} \vert^{1/2}}\leq 2^{n+1},$$
in order to deduce $\ds \left\| \sum_{\cc \in \mathfrak{C}} \langle f, \phi_{t_{\cc, 1}}  \rangle \phi_{t_{\cc, 1}}  \right\|_2 \leq \left( 2^{2n} \sum_{\cc \in \mathfrak{C}} \vert I_{\cc} \vert \right)^{1/2}$, which in turn implies inequality \eqref{eq:bessel}.
\end{proof}
\end{proposition}

Now we present the energy and size estimates for the third function:
\begin{proposition}
\label{prop:new_size_h} 
For any sequence of functions $\ds h=\lbrace h_\omega \rbrace_{\omega \in \Omega}$, we have 
\begin{equation}
\label{eq:size_est_h}
\ssize_{\rr{S}}\left( h \right) \lesssim \sup_{t \in \rr{S}} \left( \frac{1}{\vert I_t \vert}  \int_{\rr{R}} \lft \left( \sum_{\omega \in \Omega} \vert h_\omega(x) \vert^{r'}   \right)^{1/{r'}} \cdot \ci_{I_t}^M(x)  \rg^{r'} dx \right)^{1/{r'}},
\end{equation}
where $M>0$ can be chosen to be arbitrarily large, and with the implicit constant depending on $M$.
\begin{proof}
We will prove that, for any interval $I_t$, we have 
\begin{equation}
\label{eq:local_size_h}
\frac{1}{\vert I_t \vert} \int_{\rr{R}} \sum_{\omega \in \Omega} \lft \mathcal{M} \left( h_\omega \cdot \ci_{I_t}^{M}  \right)  \rg^{r'} \cdot \one_{I_t} dx \lesssim \frac{1}{\vert I_t \vert}  \int_{\rr{R}} \lft \left( \sum_{\omega \in \Omega} \vert h_\omega(x) \vert^{r'}   \right)^{1/{r'}} \cdot \ci_{I_t}^M(x)  \rg^{r'} dx.
\end{equation}
This will immediately imply \eqref{eq:size_est_h}. However, in the definition of $\ds \ssize_{\rr{S}}\left( h \right)$, we prefer to have the characteristic function $\one_{I_t}(x)$ appearing, as it makes the \emph{energy estimate} in Proposition \ref{prop:en_est_h} simpler. 

In order to prove \eqref{eq:local_size_h}, we note that
\begin{align*}
& \int_{\rr{R}} \sum_{\omega \in \Omega} \vert \mathcal{M} \left( h_\omega \cdot \ci_{I_t}^{M}  \right)  \vert^{r'} \cdot \one_{I_t} dx  \leq \sum_{\omega \in \Omega} \Big \| \mathcal{M} \left( h_\omega \cdot \ci_{I_t}^{M} \right)   \Big \|^{r'}_{r'} \\
&\lesssim \sum_{\omega \in \Omega} \Big \|  h_\omega \cdot \ci_{I_t}^{M}  \Big \|^{r'}_{r'} =\int_{\rr{R}}  \left( \sum_{\omega \in \Omega} \vert h_\omega(x) \vert^{r'} \right) \cdot \ci_{I_t}^{Mr'}dx.
\end{align*} 
Here we use the $L^{r'} \mapsto L^{r'}$ boundedness of the maximal function, so we must have $r<\infty$ and $r'>1$. The case $r=\infty$ is much easier to deal with, and it has already beed presented in Section \ref{sec:intro}.
\end{proof}
\end{proposition}

\begin{proposition}
\label{prop:en_est_h}
For any collection of tiles $\rr{S}$ and $h\in L^{r'}(\ell^{r'}(\Omega))$, we have
\[
\eenergy_{\rr{S}}\left(h \right) \lesssim \Big \| \left( \sum_{\omega \in \Omega} \vert h_\omega   \vert^{r'} \right)^{1/{r'}} \Big \|_{r'}.
\]
\begin{proof}
Let $n$ and $\mathfrak{T}$ be maximizers in \eqref{eq:def:en_h}, and for simplicity assume that $\mathfrak{T}$ is a collection of mutually disjoint columns. Then we have
\begin{align*}
&\left( \eenergy_{\rr{S}} \left( h\right) \right)^{r'} \lesssim \sum_{\mathcal{T} \in \mathfrak{T}} \int_{\rr{R}} \sum_{\omega \in \Omega(\mathcal{T})} \lft \mathcal{M}\left( h_\omega \cdot \ci_{I_{\mathcal{T}}}^M \right)(x)  \rg^{r'} \cdot \one_{I_{\mathcal{T}}}dx \\
&=\int_{\rr{R}} \sum_{\omega \in \Omega} \lft\mathcal{M}\left( h_\omega \right)(x)  \rg^{r'} \cdot \left( \sum_{\mathcal{T},\ \omega \in \Omega \left(\mathcal{T} \right)}  \one_{I_{\mathcal{T}}}\right) dx. 
\end{align*}
Here we used the inequality $\ds \mathcal{M}\left( f  \cdot \ci_I \right) \lesssim \mathcal{M}\left( f\right)$. We employ now the disjointness of the columns: if $\omega \in \Omega(\mathcal{T}_1)$ and $\omega \in \Omega(\mathcal{T}_2)$, then the tops must be disjoint in space and hence
\[
\sum_{\mathcal{T},\ \omega \in \Omega \left(\mathcal{T} \right)}  \one_{I_{\mathcal{T}}}(x) \leq 1 \enskip \text{for a. e } x.
\]
We have 
\begin{align*}
&\left( \eenergy_{\rr{S}} \left( h\right) \right)^{r'} \lesssim \int_{\rr{R}} \sum_{\omega \in \Omega} \lft\mathcal{M}\left( h_\omega \right)(x)  \rg^{r'} dx =\sum_{\omega \in \Omega} \big \| \mathcal{M}\left( h_\omega \right) \big \|_{r'}^{r'} \\
& \lesssim \sum_{\omega \in \Omega} \big \|  h_\omega \big \|_{r'}^{r'} =\Big \|  \left(  \sum_{\omega \in \Omega} \vert h_\omega \vert^{r'} \right)^{1/{r'}} \Big \|_{r'}^{r'},
\end{align*}
after making use of the $L^{r'}$-boundedness of the maximal operator.
\end{proof}
\end{proposition}

\subsection{Decomposition Lemmas and Summation of Columns/Rows}
\label{sec:decomposition}
~\\
Throughout this section, we fix the collection of tiles $\rr{S}$ and we will use the notation
\begin{align*}
\label{eq:def_global_sz_en}
&S_1:=\ssize_{\rr{S}}\left( f \right), \enskip E_1:=\eenergy_{\rr{S}}\left( f \right), \enskip S_2:=\ssize_{\rr{S}}\left( g \right), \enskip E_2:=\eenergy_{\rr{S}}\left( g \right),\\
& S_3:=\ssize_{\rr{S}}\left( h \right), \enskip E_3:=\eenergy_{\rr{S}}\left( h\right)
\end{align*}
for the ``global" sizes and energies. Using stopping times, we can partition $\rr{S}$ into smaller subcollections, on each of which we have better control on the ``local" sizes end energies.

\begin{lemma}
\label{dec_lemma_f}
Let $\ds \rr{S}' \subseteq \rr{S}$ be a subcollection of $\rr{S}$ and assume that $\ds \ssize_{\rr{S}'}\left( f\right) \leq 2^{-n_0} E_1$. Then one can partition $\ds \rr{S}'=\rr{S}'' \cup \rr{S}'''$, where 
\begin{equation}
\label{eq:dec_lemma_f_sz}
\ssize_{\rr{S}''} \left( f \right) \leq 2^{-n_0-1} E_1
\end{equation}
and $\rr{S}'''$ can be written as a union of mutually disjoint columns (in the sense of Definition \ref{def:mutually-disj}) $\ds \rr{S}'''=\bigcup_{\cc \in \mathfrak{C}} \cc$, the tops of which satisfy
\begin{equation}
\label{eq:dec_lemma_f}
\sum_{\cc \in \mathfrak{C}} \vert I_{\cc} \vert \lesssim 2^{2n_0}.
\end{equation}
\begin{proof}
We begin the decomposition algorithm by looking for tiles $s \in \rr{S}'$ which satisfy
\begin{equation}
\label{eq:dec_lemma_1}
\frac{\vert \langle f, \phi_{s_1} \rangle \vert}{\vert I_s \vert^{1/2}} > 2^{-n_0-1} E_1.
\end{equation}
If there are no such tiles, then $\ssize_{\rr{S}'} \leq 2^{-n_0-1} E_1$ and we set $\ds \rr{S}''=\rr{S}'$, $\ds \rr{S}'''=\emptyset$. 

Otherwise, start with $\rr{S}'''=\emptyset$. Among the tiles in $\ds \rr{S}'$ satisfying \eqref{eq:dec_lemma_1}, choose $s$ which has the largest spacial interval $I_s$(and hence the smallest frequency interval $\omega_1$), and so that both $I_s$ and $\omega_{s_1}$ are situated leftmost. Then construct the column
\[
\cc_1:=\lbrace t \in \rr{S}': I_t \subseteq I_s, \enskip \text{and } \enskip \omega_{t_1} \subseteq \omega_{s_1}\rbrace.
\] 
Now set $\ds \rr{S}''':=\rr{S}''' \cup \cc_1$, $\rr{S}'=\rr{S}' \setminus \rr{S}'''$, and restart the algorithm.

At the end, we will have a collection of columns $\cc_1, \ldots, \cc_N$ which constitute $\rr{S}'''$, and $\rr{S}''$, in which none of the tiles satisfies \eqref{eq:dec_lemma_1}. The columns are disjoint by construction, so we are left with proving the inequality \eqref{eq:dec_lemma_f}, which follows directly from the energy definition. For the columns $\cc_1, \ldots, \cc_N$, we know that their tops $s_1, \ldots, s_N$ satisfy
\[
 \frac{\vert \langle f, \phi_{s_j, 1}  \rangle \vert}{\vert I_s \vert^{1/2}} \geq 2^{-n_0-1}E_1, \enskip \text{for all  } 1\leq j \leq N. 
\]
Hence $\ds \left( E_1 2^{- n_0-1} \right)^2 \sum_{j} \vert I_{s_j} \vert \leq \left( \eenergy_{\rr{S}'} \left(f \right) \right)^2 \leq \left( \eenergy_{\rr{S}} \left(f\right) \right)^2=E_1^2$.
\end{proof}
\end{lemma}

A very similar result holds for $g$, with columns being replaced by rows:
\begin{lemma}
\label{dec_lemma_g}
Let $\ds \rr{S}' \subseteq \rr{S}$ be a subcollection of $\rr{S}$ and assume that $\ds \ssize_{\rr{S}'}\left( g\right) \leq 2^{-n_0} E_2$. Then one can partition $\ds \rr{S}'=\rr{S}'' \cup \rr{S}'''$, where 
\begin{equation}
\label{eq:dec_lemma_g_sz}
\ssize_{\rr{S}''} \left( g \right) \leq 2^{-n_0-1} E_2
\end{equation}
and $\rr{S}'''$ can be written as a union of mutually disjoint rows $\ds \rr{S}'''=\bigcup_{\ii{R} \in \mathfrak{R}} \ii{R}$, the tops of which satisfy
\begin{equation}
\label{eq:dec_lemma_g}
\sum_{\ii{R} \in \mathfrak{R}} \vert I_{\ii{R}} \vert \lesssim 2^{2n_0}.
\end{equation}
\end{lemma}

We have seen already that the size of $h$ depends both on columns and rows. This behavior will also be displayed in the decomposition lemma for $h$:
\begin{lemma}
\label{dec_lemma_h}
Let $\ds \rr{S}' \subseteq \rr{S}$ be a subcollection of $\rr{S}$ and assume that $\ds \ssize_{\rr{S}'}\left( h\right) \leq 2^{-n_0} E_3$. Then one can partition $\ds \rr{S}'=\rr{S}'' \cup \rr{S}'''$, where 
\begin{equation}
\label{eq:dec_lemma_h_sz}
\ssize_{\rr{S}''} \left( h \right) \leq 2^{-n_0-1} E_3
\end{equation}
and $\rr{S}'''$ can be written as the union of $\mathfrak{C}$, a collection of mutually disjoint columns, and $\mathfrak{R}$, a collection of mutually disjoint rows: $\ds \rr{S}'''=\bigcup_{\cc \in \mathfrak{C}} \cc \cup \bigcup_{\ii{R} \in \mathfrak{R}} \ii{R}$. Moreover, we have
\begin{equation}
\label{eq:dec_lemma_h}
\sum_{\cc \in \mathfrak{C}} \vert I_{\cc} \vert \lesssim 2^{r' n_0} \enskip \text{and  }\enskip \sum_{\ii{R} \in \mathfrak{R}} \vert I_{\ii{R}} \vert \lesssim 2^{r' n_0}.
\end{equation}
\begin{proof}
The proof will be similar to that of Lemma \ref{dec_lemma_f}. 
We initialize $\rr{S}''=\rr{S}'''=\emptyset$, and we begin by looking for ``extremizers" for $\ssize_{\rr{S}'} \left( h \right)$. That is, we  look for columns $\cc \subseteq \rr{S}'$ satisfying
\begin{equation}
\label{eq:dec_lemma_h_col}
\left( \frac{1}{\vert I_{\cc} \vert} \int_{\rr{R}} \sum_{\omega \in \Omega(\cc)} \Big \vert \mathcal{M}\left( h_\omega \cdot \ci_{I_{\cc}}^{M}  \right) \Big \vert^{r'} \cdot \one_{I_{\cc}} dx \right)^{1/{r'}} >2^{-n_0-1} E_3.
\end{equation}
If there are no such columns, we search for rows $\ii{R}\subseteq \rr{S}'$ which satisfy
\begin{equation}
\label{eq:dec_lemma_h_row}
\left( \frac{1}{\vert I_{\ii{R}} \vert} \int_{\rr{R}} \sum_{\omega \in \Omega(\ii{R})} \Big \vert \mathcal{M}\left( h_\omega \cdot \ci_{I_{\ii{R}}}^{M}  \right) \Big \vert^{r'} \cdot \one_{I_{\ii{R}}} dx \right)^{1/{r'}} >2^{-n_0-1} E_3.
\end{equation}
When there are no more columns or rows satisfying \eqref{eq:dec_lemma_h_col} or \eqref{eq:dec_lemma_h_row}, set $\rr{S}''=\rr{S}$, which will have $\ds \ssize_{\rr{S}''} \left( h\right) \leq 2^{-n_0-1}$.

Instead, if we have columns satisfying \eqref{eq:dec_lemma_h_col}, we select the ones which are maximal with respect to inclusion,  have the largest spatial top interval $I_s$, and among these, we choose the one whose frequency interval $\omega_{s, 1}$ and spatial interval $I_s$ are leftmost. Ultimately, we want to obtain a collection $\mathfrak{C}$ of disjoint columns. Let $\cc_{1}$ be such a column, and denote $s_1$ its top. Note that a tile $t$ satisfying $\omega_{t, 1} \subset \omega_{s_1, 1}$ and $I_{s_1} \subset I_t$  cannot be the top of a column satisfying \eqref{eq:dec_lemma_h_col}, for it should have been selected first. 

Then we set $\ds\rr{S}'''=\rr{S}''' \cup \cc_1$ and $\rr{S}':= \rr{S}'\setminus \cc_1$, and repeat the algorithm. That is, we search for columns in the updated $\rr{S}'$ satisfying \eqref{eq:dec_lemma_h_col}, obtaining eventually a collection $\mathfrak{C}=\cc_1\cup \ldots \cup\cc_N$ of mutually disjoint columns, with disjoint tops, satisfying \eqref{eq:dec_lemma_h_col}.
Following that, we repeat the same procedure, obtaining a collection $\mathfrak{R}= \ii{R}_1\cup \ldots \cup \ii{R}_{\tilde{N}}$ of rows satisfying \eqref{eq:dec_lemma_h_row}. We will have $\ds \rr{S}'''=\bigcup_{\cc \in \mathfrak{C}} \cc \cup \bigcup_{\ii{R} \in \mathfrak{R}} \ii{R}$. Also, $\rr{S}''$ consists of the tiles in $\rr{S}' \setminus \rr{S}''' $, and will have the property that $\ds \ssize_{\rr{S}''} \left( h\right) \leq 2^{-n_0-1}$.

Then \eqref{eq:dec_lemma_h} follows from Definition \ref{def:en_h}, similarly to the proof in Lemma \ref{dec_lemma_f}.
\end{proof}
\end{lemma}

Simultaneously applying the decomposition results above, and re-iterating until all tiles in $\rr{S}$ are exhausted, we obtain a splitting of $\rr{S}$ into collections of columns and rows.
\begin{proposition}
\label{lemma_splitting}
One can write $\rr{S}$ as $\ds \rr{S}=\bigcup_{n \in \rr{Z}} \rr{S}_n^1 \cup \rr{S}_n^2$, where each $\rr{S}_n^1$ is a union of disjoint columns, and each $\rr{S}_n^2$ is a union of disjoint rows, for which we have:
\begin{itemize}
\item[(a)]for $i\in\{1,2\}$ then $\ds \ssize_{\rr{S}_n^i}\left( f \right) \leq \min \left( 2^{-n} E_1 , S_1\right)$,
\item[(b)]for $i\in\{1,2\}$ then  $\ds \ssize_{\rr{S}_n^i} \left( g \right) \leq \min \left( 2^{-n} E_2 , S_2\right)$,
\item[(c)]for $i\in\{1,2\}$ then  $\ds \ssize_{\rr{S}_n^i} \left( h \right) \leq \min \left( 2^{-\frac{2n}{r'}} E_3 , S_3\right)$ and
\item[(d)]$\ds \sum_{\cc \in \rr{S}_n^1} \vert I_{\cc} \vert \lesssim 2^{2n} \enskip \text{and  }\enskip \sum_{\ii{R} \in \rr{S}_n^2} \vert I_{\ii{R}} \vert \lesssim 2^{2n}$.
\end{itemize}
Moreover, $\rr{S}_n^1$ is nonempty if and only if one of the following holds:
\[
2^{-n-1}E_1 \leq \ssize_{\rr{S}_n^1} \left( f \right) \leq \min \left( 2^{-n} E_1 , S_1\right), \enskip \text{or  }\enskip 2^{-\frac{2\left(n-1 \right)}{r'}}E_3 \leq \ssize_{\rr{S}_n^1} \left( h \right) \leq \min \left( 2^{-\frac{2n}{r'}} E_3 , S_3\right).
\]
Similarly,  $\rr{S}_n^2$ is nonempty if and only if
\[
2^{-n-1}E_2 \leq \ssize_{\rr{S}_n^2} \left( g \right) \leq \min \left( 2^{-n} E_2 , S_2\right), \enskip \text{or  }\enskip 2^{-\frac{2\left(n-1 \right)}{r'}}E_3 \leq \ssize_{\rr{S}_n^2} \left( h \right) \leq \min \left( 2^{-\frac{2n}{r'}} E_3 , S_3\right).
\]
\end{proposition}

\section{Generic Estimate for the trilinear form $\Lambda_{\rr{S}}(f,g,h)$ }
\label{sec:gen-est}
Using Proposition \ref{lemma_splitting}, we obtain a way of estimating the trilinear form $\ds \Lambda_{\rr{S}}(f,g,h)$ by using the sizes and energies. We recall that $\alpha$ was defined as $\ds \alpha=\frac{1}{2}-\frac{1}{r}$.
\begin{proposition}
\label{prop:generic_est}
If $F, G, H'$ and $f, g, h=\lbrace h_\omega\rbrace_\omega$ are as in \eqref{conditions_F,G,H}, then
\begin{align}
\label{generic_est}
&\big \vert \Lambda_{\rr{S}}(f, g, h)\big \vert \lesssim \left( \sup_{s \in \rr{S}} \frac{1}{\vert I_s \vert} \int_{\rr{R}} \one_{G} \cdot \ci_{I_s}^{100} dx  \right)^{1/r} \cdot S_1^{4 \alpha \theta_1} E_1^{1-4\alpha \theta_1} S_2^{4 \alpha \theta_2} E_2^{2\alpha -4 \alpha \theta_2} S_3^{\frac{r'}{2} \cdot 4 \alpha \theta_3} E_3^{1-\frac{r'}{2} \cdot 4 \alpha \theta_3} \\
&+\left( \sup_{s \in \rr{S}} \frac{1}{\vert I_s \vert} \int_{\rr{R}} \one_{F} \cdot \ci_{I_s}^{100} dx  \right)^{1/r} \cdot S_1^{4 \alpha \beta_1} E_1^{2 \alpha-4\alpha \beta_1} S_2^{4 \alpha \beta_2} E_2^{1-4 \alpha \beta_2} S_3^{\frac{r'}{2} \cdot 4 \alpha \beta_3} E_3^{1-\frac{r'}{2} \cdot 4 \alpha \beta_3},\nonumber
\end{align}
whenever the variables $\theta_j, \beta_j$ satisfy $\theta_1+\theta_2+\theta_3=1$, $\beta_1+\beta_2+\beta_3=1$ and
\begin{equation}
\label{eq:cond_theta,beta}
0 \leq \theta_1, \beta_2 \leq \min\left( 1, \frac{1}{4 \alpha}\right), \enskip 0 \leq \theta_2, \beta_1 \leq \frac{1}{2}, 0<\theta_3, \beta_3 \leq 1.
\end{equation}
\begin{proof}
From Proposition \ref{lemma_splitting}, we have
\begin{align*}
\Lambda_{\rr{S}}(f, g, h)=\sum_{n \in \rr{Z}} \left( \sum_{\cc \in \rr{S}_n^1} \Lambda_{\cc}(f, g, h)+  \sum_{\ii{R} \in \rr{S}_n^2} \Lambda_{\ii{R}}(f, g, h)  \right), 
\end{align*}
and from Proposition \ref{prop:column est}, for any $\cc \in \rr{S}_n^1$,
\begin{equation}
\label{eq:gen_est_1}
\big \vert \Lambda_{\cc}(f, g, h)\big \vert \lesssim \left( \sup_{s \in \rr{S}} \frac{1}{\vert I_s \vert} \int_{\rr{R}} \one_{G} \cdot \ci_{I_s}^{100} dx  \right)^{1/r}   \ssize_{\rr{S}_n^1} \left( f\right) \cdot \left( \ssize_{\rr{S}_n^1} \left( g\right)  \right)^{2\alpha}
\cdot \ssize_{\rr{S}_n^1} \left( h\right) \cdot \vert I_{\cc} \vert.
\end{equation}
Here we used the fact that along a column $\cc$, the frequency intervals $\ds\lbrace \omega_{s_2} \rbrace_{s \in \cc}$ are disjoint, and orthogonality implies that
\begin{equation*}
\left( \frac{1}{\vert I_{\cc} \vert} \sum_{s \in \cc} \vert \langle g, \phi_{s_2} \rangle \vert^2 \right)^{1/r} \lesssim \left( \frac{1}{\vert I_{\cc} \vert}  \| g \cdot \ci_{I_c}^{100} \|_2^2\right)^{1/r} \lesssim \left( \sup_{s \in \rr{S}} \frac{1}{\vert I_s \vert} \int_{\rr{R}} \one_{G} \cdot \ci_{I_s}^{100} dx  \right)^{1/r}.
\end{equation*}
It will be enough to estimate
\begin{equation}
\label{eq:est_1}
\sum_{n \in \rr{Z}} \ \sum_{\cc \in \rr{S}_n^1} \ssize_{\rr{S}_n^1}\left( f \right) \cdot \left( \ssize_{\rr{S}_n^1}\left( g \right) \right)^{2 \alpha} \cdot \ssize_{\rr{S}_n^1}\left( h \right) \cdot \vert I_{\cc} \vert  \tag{I}
\end{equation}
and correspondingly,
\begin{equation}
\label{eq:est_2}
\sum_{n \in \rr{Z}} \  \sum_{\ii{R} \in \rr{S}_n^2}  \left( \ssize_{\rr{S}_n^2}\left( f \right) \right)^{2 \alpha} \ssize_{\rr{S}_n^2}\left( 2 \right)\cdot \ssize_{\rr{S}_n^2}\left( h \right) \cdot \vert I_{\ii{R}} \vert  \tag{II}.
\end{equation}
For \eqref{eq:est_1}, Proposition \ref{lemma_splitting} yields 
\begin{align*}
\eqref{eq:est_1} &\lesssim \sum_{n \in \rr{Z}} \min \left( 2^{-n} E_1, S_1 \right) \cdot \left( \min \left( 2^{-n} E_2, S_2 \right) \right)^{2 \alpha} \cdot \min \left( 2^{-\frac{2n}{r'}} E_3, S_3 \right) \cdot 2^{2n}\\
& =\sum_{n \in \rr{Z}} E_1 E_2^{2 \alpha} E_3 \min \left( 2^{-n},\frac{S_1}{E_1} \right) \cdot \left(  \min \left( 2^{-n},\frac{S_2}{E_2} \right)\right)^{2 \alpha} \cdot \min \left( 2^{-\frac{2n}{r'}},\frac{S_3}{E_3} \right)\cdot 2^{2n}.
\end{align*}
From Proposition \ref{lemma_splitting}, we know that the collections $\rr{S}_n^1$ are non-empty as long as $\ds 2^{-n} \leq \max \left( \frac{S_1}{E_1}, \left(\frac{S_3}{E_3}\right)^{\frac{r'}{2}} \right)$. Since the expression above displays no symmetries in the sizes for $f, g$, and $h$, one needs to analyze separately all the possibilities: $$\ds \frac{S_1}{E_1} \leq \frac{S_2}{E_2} \leq \left(\frac{S_3}{E_3}\right)^{\frac{r'}{2}}, \frac{S_2}{E_2} \leq \frac{S_1}{E_1} \leq \left(\frac{S_3}{E_3}\right)^{\frac{r'}{2}}, \text{  etc...}$$
We will illustrate only the first case, the others being routine repetitions.
So assume that 
\begin{equation} \frac{S_1}{E_1} \leq \frac{S_2}{E_2} \leq \left(\frac{S_3}{E_3}\right)^{\frac{r'}{2}}.
\label{case1}
\end{equation}
We split \eqref{eq:est_1} into several sub-sums according to $2^{-n}$, but each of them will still be denoted by \eqref{eq:est_1} for simplicity.
\begin{itemize}
\item[(i)] Case where $\ds 2^{-n} \leq \frac{S_1}{E_1} \leq \frac{S_2}{E_2} \leq \left(\frac{S_3}{E_3}\right)^{\frac{r'}{2}}$. Then 
\begin{align*}
\eqref{eq:est_1} \lesssim E_1E_2^{2 \alpha} E_3 \sum_{n} 2^{-n} 2^{-2\alpha n} 2^{-\frac{2n}{r'}} 2^{2n} \lesssim E_1E_2^{2 \alpha} E_3 \sum_n 2^{-n\left( 1+2 \alpha + \frac{2}{r'}-2\right)}.
\end{align*}
With the observation that $\ds 1+2 \alpha + \frac{2}{r'}-2=4 \alpha$, and under the assumption that $\ds 2^{-n} \leq \min \left( \frac{S_1}{E_1}, \frac{S_2}{E_2}, \left( \frac{S_3}{E_3} \right)^{\frac{r'}{2}}   \right)$, we have
\begin{align*}
\eqref{eq:est_1} \lesssim E_1E_2^{2 \alpha} E_3 \left( \frac{S_1}{E_1} \right)^{4 \alpha \theta_1} \cdot \left( \frac{S_2}{E_2} \right)^{4 \alpha \theta_2} \cdot \left( \frac{S_3}{E_3} \right)^{\frac{r'}{2} \cdot 4 \alpha \theta_3},
\end{align*}
where $0 \leq \theta_1, \theta_2, \theta_3 \leq 1$, and $\theta_1+\theta_2+\theta_3=1$. This further implies the desired expression from \eqref{generic_est}.
\item[(ii)] If $\ds \frac{S_1}{E_1} \leq 2^{-n}\leq \frac{S_2}{E_2} \leq \left(\frac{S_3}{E_3}\right)^{\frac{r'}{2}}$, then
\begin{align*}
\eqref{eq:est_1} \lesssim E_1E_2^{2 \alpha} E_3 \frac{S_1}{E_1}  \sum_{n} 2^{-n\left( 1-\frac{4}{r}\right)},
\end{align*}
and we have to consider two possibilities: $\ds 1-\frac{4}{r}\geq 0$ and $\ds 1 -\frac{4}{r}<0$.
\begin{itemize}
\item[(a)] If $\ds 1-\frac{4}{r} \geq 0$, then 
\begin{align*}
\eqref{eq:est_1} \lesssim  E_1E_2^{2 \alpha} E_3 \frac{S_1}{E_1} \left( \frac{S_2}{E_2} \right)^{1-\frac{4}{r}} \lesssim E_1E_2^{2 \alpha} E_3 \cdot \left( \frac{S_1}{E_1} \right)^{4 \alpha \theta_1} \cdot \left( \frac{S_1}{E_1} \right)^{1-4 \alpha \theta_1} \cdot  \left( \frac{S_2}{E_2} \right)^{1-\frac{4}{r}}.
\end{align*}
As long as $1-4 \alpha \theta_1 \geq 0$, we obtain 
\begin{align*}
\eqref{eq:est_1} \lesssim E_1E_2^{2 \alpha} E_3 \cdot \left( \frac{S_1}{E_1} \right)^{4 \alpha \theta_1} \left( \frac{S_2}{E_2} \right)^{1-\frac{4}{r}+1 -4 \alpha \theta_1}=E_1E_2^{2 \alpha} E_3 \cdot \left( \frac{S_1}{E_1} \right)^{4 \alpha \theta_1} \left( \frac{S_2}{E_2} \right)^{4 \alpha -4 \alpha \theta_1}.
\end{align*}
This implies the estimate \eqref{generic_est}, since $4 \alpha -4\alpha \theta_1=4\alpha \theta_2+4\alpha \theta_3$, with $\theta_1, \theta_2, \theta_3$ positive and adding up to $1$.
\item[(b)] On the other hand, if $\ds 1-\frac{4}{r}<0$, then $\ds \sum_{n} 2^{-n \left(1-\frac{4}{r}\right)} \lesssim \left(\frac{S_1}{E_1} \right)^{1-\frac{4}{r}}$ and 
\begin{align*}
\eqref{eq:est_1} \lesssim E_1E_2^{2 \alpha} E_3 \cdot \left( \frac{S_1}{E_1} \right)^{ 2-\frac{4}{r}}=E_1E_2^{2 \alpha} E_3 \cdot \left( \frac{S_1}{E_1} \right)^{4 \alpha}.
\end{align*}
This immediately implies \eqref{generic_est}.
\item[(c)] If $\ds 1-\frac{4}{r}=0$, a similar estimate is obtained by an easy interpolation between $(a)$ and $(b)$.
\end{itemize}
\item[(iii)] The last case we present here is $\ds \frac{S_1}{E_1} \leq \frac{S_2}{E_2} \leq  2^{-n} \leq \left(\frac{S_3}{E_3}\right)^{\frac{r'}{2}}$. In this situation, 
\begin{align*}
\eqref{eq:est_1} \lesssim E_1E_2^{2 \alpha} E_3 \frac{S_1}{E_1} \left( \frac{S_2}{E_2}\right)^{2 \alpha} \sum_{n} 2^{-n \left( \frac{2}{r'}-2 \right)}
\end{align*}
The exponent $\dfrac{2}{r'}-2$ is negative; in fact $\dfrac{2}{r'}-2=-\dfrac{2}{r}$, hence $$\ds  \sum_{n} 2^{-n \left( \frac{2}{r'}-2 \right)} \lesssim \left( \frac{S_2}{E_2}\right)^{-\frac{2}{r}}.$$
It follows that 
\begin{align*}
\eqref{eq:est_1} &\lesssim E_1E_2^{2 \alpha} E_3 \cdot \left( \frac{S_1}{E_1} \right)^{4 \alpha \theta_1} \left( \frac{S_1}{E_1} \right)^{1-4 \alpha \theta_1} \left( \frac{S_2}{E_2} \right)^{4 \alpha \theta_2} \left( \frac{S_2}{E_2} \right)^{2 \alpha -4 \alpha \theta_2-\frac{2}{r}} \\
&\lesssim E_1E_2^{2 \alpha} E_3 \cdot \left( \frac{S_1}{E_1} \right)^{4 \alpha \theta_1} \left( \frac{S_2}{E_2} \right)^{4 \alpha \theta_2} \left( \frac{S_2}{E_2} \right)^{1-4 \alpha \theta_1 +2 \alpha -4 \alpha \theta_2-\frac{2}{r}} \\
&=E_1E_2^{2 \alpha} E_3 \cdot \left( \frac{S_1}{E_1} \right)^{4 \alpha \theta_1} \left( \frac{S_2}{E_2} \right)^{4 \alpha \theta_2} \left( \frac{S_2}{E_2} \right)^{4 \alpha- 4 \alpha \theta_1-4 \alpha \theta_2}.
\end{align*}
From the last identity we get the conclusion. Here again we need the assumption $1-4 \alpha \theta_1 \geq 0$.
\end{itemize}
This concludes the proof of the estimate for \eqref{eq:est_1} in the case where \eqref{case1} holds. \\
The rest of the cases for estimating $\eqref{eq:est_1}$, as well as the estimates for \eqref{eq:est_2} reduce to similar computations, and for that reason we don't present the details here.
\end{proof}
\end{proposition}

\section{Localization of sizes and energies}
\label{sec-localization}
If we apply Proposition \ref{prop:generic_est} directly, the range that we obtain for $T_r$ is restricted by the conditions
\begin{equation}
\label{eq:restricted_Range}
r'< p \leq r, \enskip r'<q \leq r. 
\end{equation}  
To obtain a larger range, we will need to use local variants of the previous Propositions \ref{propo:en_est_f} and \ref{prop:en_est_h}.

Let $I_0$ be a fixed dyadic interval. We denote $\rr{S}(I_0)$ the subcollection of tiles with spatial interval contained inside $I_0$:
\begin{equation}
\rr{S}(I_0):=\lbrace s \in \rr{S} : I_s \subseteq I_0  \rbrace.
\end{equation} 
We have the following improvements for the \emph{energies} on $\rr{S}(I_0)$:
\begin{proposition}[modification of Prop. \ref{propo:en_est_f}]
\[
\eenergy_{\rr{S}(I_0)}\left( f \right) \lesssim \| f \cdot \ci_{I_0} \|_2, \enskip \eenergy_{\rr{S}(I_0)}\left( g \right) \lesssim \| g \cdot \ci_{I_0} \|_2.
\]
\end{proposition}

\begin{proposition}[modification of Prop. \ref{prop:en_est_h}]
Similarly, $$\ds \eenergy_{\rr{S}(I_0)} \left( h\right) \lesssim \| \left( \sum_{\omega \in \Omega} \vert h_\omega \vert^{r'}   \right)^{1/{r'}} \cdot \ci_{I_0}^M \|_{r'}.$$
\begin{proof}
All the tiles in $\rr{S}(I_0)$ are so that $I_s \subseteq I_0$; so in particular, if $\lbrace \mathcal{T} \rbrace_{\mathcal{T} \in \mathfrak{T}}$ is a collection of disjoint columns or rows which is a maximizer for $\eenergy_{\rr{S}(I_0)} \left( h\right)$, then $\ds \ci_{I_{\mathcal{T}}} \leq \ci_{I_0}$ for all $\mathcal{T} \in \mathfrak{T}$.
The desired estimate follows easily from the observation that
\[
\left(\eenergy_{\rr{S}(I_0)} \left( h\right)   \right)^{r'} \lesssim \sum_{\mathcal{T} \in \mathfrak{T}} \int_{\rr{R}} \sum_{\omega \in \Omega(\mathcal{T})} \big \vert \mathcal{M}\left( h_\omega \cdot \ci_{I_0}^{M}   \right)(x)  \big \vert^{r'} \cdot \one_{I_\mathcal{T}} dx.
\]
 A reasoning similar to that in Proposition \ref{prop:en_est_h} yields that
 \[
 \eenergy_{\rr{S}(I_0)}\left( h\right) \lesssim \left( \sum_{\omega \in \Omega} \| \mathcal{M}\left( h_\omega \cdot \ci_{I_0}^M\right) \|_{r'}^{r'} \right)^{1/{r'}} \lesssim \Big \|  \left( \sum_{\omega \in \Omega} \vert h_\omega   \vert^{r'}   \right)^{1/r'} \cdot \ci_{I_0}^M  \Big \|_{r'}.
 \]

\end{proof}
\end{proposition}

\section{Proof of Theorem \ref{thm:main}}
\label{sec:main-proof}

\begin{proof}[Proof of Theorem \ref{thm:main}]
Now we are ready to provide a proof for our main result. To start with, we will partition the collection $\ds \rr{S}:=\bigcup_{d \geq 0} \rr{S}_d$, and for each of these subcollections we will show an inequality similar to \eqref{eq:main:dual_form} of Proposition \ref{prop:main}:
\begin{equation}
\label{eq:prop_main_decay}
\big \vert \Lambda_{\rr{S}_d}(f, g, h)  \big \vert \lesssim 2^{-10d} \cdot\vert F\vert^{\nu_1}\cdot \vert G\vert^{\nu_2}\cdot \vert H\vert^{\nu_3},
\end{equation}
where $\nu_1+\nu_2+\nu_3=1$, and $\ds \left(\nu_1, \nu_2, \nu_3 \right)$ is in a small neighborhood of $\ds \left(\frac{1}{p}, \frac{1}{q}, \frac{1}{s'}\right)$.

Given measurable sets $F, G, H$ with $\vert H \vert=1$, we define the exceptional set as
\begin{equation}
\label{eq:def_exceptional_Set}
\mathcal{E} :=\lbrace x: \mathcal{M}\left( \one_F\right)(x) > C \vert F \vert  \rbrace \cup \lbrace x: \mathcal{M}\left( \one_G\right)(x) > C \vert G \vert  \rbrace.
\end{equation}
For a constant $C$ large enough, we have $\vert \mathcal{E} \vert \ll 1$, so $H':=H \setminus \mathcal{E}$ is going to be a major subset of $H$. Let $f, g, \lbrace h_\omega \rbrace_{\omega \in \Omega}$ be so that
\[
\vert f(x)\vert \leq \one_{F}(x), \enskip \vert g(x)\vert \leq \one_{G}(x), \enskip \left( \sum_\omega \vert h_\omega \vert^{r'}  \right)^{1/r'} \leq \one_{H'} \enskip \text{for a. e. $x$.}
\]
Then the subcollections $\rr{S}_d$ which constitute the partition $\ds \rr{S}:=\bigcup_{d \geq 0} \rr{S}_d$, are defined by
\[
\rr{S}_d:=\lbrace s \in \rr{S}:  2^{d} \leq 1+\frac{\dist\left(I_s, \mathcal{E}^c \right)}{\vert I_s \vert} \leq 2^{d+1}  \rbrace.
\]
In order to keep things simple, we temporarily suppress the $d$-dependency in the notation $\rr{S}_d$.

Next, we will use Proposition \ref{prop:generic_est}, applied to some subcollections $\rr{S}_{n_1, n_2, n_3}(I_0) \subseteq \rr{S}(I_0)$ for suitable intervals $I_0$. The proof will become rather technical, so we will try to present the main ideas before going forward  with the details. As mentioned before, applying Proposition \ref{prop:generic_est} to $\rr{S}$  or even $\rr{S}_d$ will yield a range of boundedness for $T_r$ which is not optimal. A similar situation appears in the case of the bilinear Hilbert transform operator $BHT$ defined in \eqref{def:BHT}. Using \emph{sizes} and \emph{energies}, one can only obtain the $L^p \times L^q \to L^s$ boundedness of $BHT$ for $p, q, s$ satisfying
\[
\Big \vert \frac{1}{p}-\frac{1}{q} \Big \vert < \frac{1}{2}, \enskip \text{and} \enskip \frac{2}{3}< s < 2.
\]
One gets a larger range for $BHT$ by interpolating between the adjoint operators, and using the symmetries of the trilinear form. The procedure is described in \cite{biest}, or \cite{multilinear_harmonic}.

The trilinear form associated to $T_r$ however, lacks symmetry in $f, g$, and $h$. Instead, we will use local estimates that in turn will allow us to represent the \emph{energy} as an average over \emph{certain} intervals. The selection of the intervals is done through three stopping times, with respect to $f$, $g$, and $h$. The idea of using local estimates in order to convert the energies into averages originates from \cite{vv_BHT}.

We will obtain $\ii{I}_1^{n_1}, \ii{I}_2^{n_2}, \ii{I}_3^{n_3}$, three collections of dyadic intervals indexed after the set of natural numbers. If $I_0 \in \ii{I}_1^{n_1}$, then 
\[
2^{-n_1-1} \leq \frac{1}{\vert I_0 \vert} \int_{\rr{R}} \one_F \cdot \ci_{I_0} dx \leq 2^{-n_1}.
\]
Moreover, for every interval $I_0 \in \ii{I}_1^{n_1}$, we will have a corresponding collection $\rr{S}_{n_1}^1(I_0) \subseteq \rr{S}$, which will be constructed in the stopping time. For every $J \subseteq I_0$, and any subcollection $\rr{S}' \subseteq \rr{S}_{n_1}^1(I_0)$, we have
\begin{equation}
\label{eq:unif_eq}
\max \left( \ssize_{\rr{S}'\left(J\right)}\left( f \right), \frac{\|  \one_F \cdot \ci_J \|_1}{\vert J \vert} \right) \leq 2^{-n_1}.
\end{equation}

Similarly, $\ii{I}_2^{n_2}$ and $\ii{I}_3^{n_3}$ generate partitions of $\rr{S}$:
\[
\rr{S}:=\bigcup_{n_2} \bigcup_{I_0 \in \ii{I}_2^{n_2}} \rr{S}_{n_2}^2\left( I_0\right) = \bigcup_{n_3} \bigcup_{I_0 \in \ii{I}_2^{n_3}} \rr{S}_{n_3}^3\left( I_0\right),
\]
with the only difference that for the sizes associated with $h$, we have
\[
2^{-n_3-1}\leq \frac{1}{\vert I_0 \vert} \int_{\rr{R}} \one_{H'} \cdot \ci_{I_0}^{M r'} dx \leq 2^{-n_3}, \enskip \text{and}
\]
\[
\max \left( \sssize_{\rr{S'}\left(J\right)} \left(h\right), \frac{\| \one_{H'} \cdot \ci_{J}^{M}   \|_{r'}}{\vert J \vert^{1/{r'}}} \right) \leq 2^{-\frac{n_3}{r'}}.
\]

Now we describe the selection algorithm for $\ii{I}_3^{n_3}$ and $\rr{S}_{n_3}^3(I_0)$, the construction of $\ii{I}_1^{n_1}, \rr{S}_{n_1}^1(I_0)$ and $\ii{I}_2^{n_2}, \rr{S}_{n_2}^2(I_0)$  being similar.

\subsection*{Step: Selection algorithm for $h$}~\\
For the stopping time, we will use a new version of size of $h$. Firstly, given a collection $\rr{S}$ of tiles, we denote 
\begin{equation}
\label{eq:collection_plus}
\ii{I}^+\left( \rr{S} \right):=\lbrace I \text{ dyadic interval }: \exists s \in \rr{S} \text{ so that } I_s \subseteq I \rbrace.
\end{equation}
Then define
\[
\sssize_{\rr{S}}\left( h\right):=\sup_{I\in \ii{I}^+\left(\rr{S}\right)}\left( \frac{1}{\vert I \vert} \int_{\rr{R}} \one_{H'} \cdot \ci_{I}^{Mr'} dx\right)^{1/{r'}}.
\]
While this might appear unnatural, the reason why we are defining this new size is so that we can compare $\ds\ssize_{\rr{S}(I_0)}\left( h \right),  \sssize_{\rr{S}(I_0)}\left( h \right)$ and $\ds \frac{\| \one_{H'} \cdot \ci_{I_0}^M \|_{r'}}{\vert  I_0 \vert^{1/r'}}$. Then, using localization results, we can convert the energy into a size as well.

For $n_3 \geq  1$, assume that we have constructed the collection $\ii{I}_3^{n_3-1}$ already, and for every $I_0 \in \ii{I}_3^{n_3-1}$, also the collection of tiles $\rr{S}_{n_3-1}^3(I_0)$. Then $\bar{\rr{S}}_{n_3}$ is the collection of available tiles, which has the property that
\[
\sssize_{\bar{\rr{S}}_{n_3}}\left( h \right)^{r'} \leq 2^{- n_3}.
\]
We will construct the similar sets $\ii{I}_{3}^{n_3}$, and $\rr{S}_{n_3}^3(I_0)$. First, look for intervals $\ds I \in \ii{I}^+\left(\bar{\rr{S}}_{n_3} \right)$ with the property that
\begin{equation}
\label{eq:sel_h}
2^{-n_3-1} \leq \frac{1}{\vert I \vert} \int_{\rr{R}} \one_{H'}\cdot \ci_{I}^{M r'} dx \leq 2^{-n_3}.
\end{equation}
If there are no such intervals in $\ds \ii{I}^+\left(\bar{\rr{S}}_{n_3}\right)$, set $\ii{I}_{3}^{n_3}=\emptyset$, and $\bar{\rr{S}}_{n_3+1}:=\bar{\rr{S}}_{n_3}$; continue the procedure with $n_3$ replaced by $n_3+1$.

Otherwise, pick such an interval $\ds I_0 \in \ii{I}^+\left(\bar{\rr{S}}_{n_3}\right)$ satisfying \eqref{eq:sel_h}, which is maximal with respect to inclusion. It will contain some $s \in \bar{\rr{S}}_{n_3}$. Now define
\[
\rr{S}_{n_3}^3(I_0):=\lbrace s \in \bar{\rr{S}}_{n_3}: I_s \subseteq I_0  \rbrace,
\]
and set $\ds \bar{\rr{S}}_{n_3}:=\bar{\rr{S}}_{n_3} \setminus \rr{S}_{n_3}^3(I_0)$. We continue the search for maximal dyadic intervals  satisfying \eqref{eq:sel_h}, and which are contained in $\ds \ii{I}^+\left(\rr{S}_{n_3}^3(I_0)\right)$.

In this way, given $n_3 \geq 0$, and $I_0 \in \ii{I}_3^{n_3}$, we have, for any $t \in \rr{S}_{n_3}^3(I_0)$
\[
\frac{1}{\vert I_t \vert} \int_{\rr{R}} \one_{H'}\cdot \ci_{I_t}^{Mr'} dx \leq 2^{-n_3},
\]
for otherwise the tile $t$ would have been chosen in some $\rr{S}_3^{m_3}$ for some $m_3<n_3$. For similar reasons, if $I_0 \in \ii{I}_3^{n_3}$, $J \subseteq I_0$, and there exists at least one $t \in \rr{S}_{n_3}^3(I_0)$ with $I_t \subseteq J$, then
\[
\frac{1}{\vert J \vert}\int_{\rr{R}} \one_{H'}\cdot \ci_{J}^{Mr'} dx \leq 2^{-n_3}.
\]
This is due to the stopping time algorithm. If there are no more intervals $\ds I \in \ii{I}^+\left(\rr{S}_{n_3}^3 \right)$ satisfying \eqref{eq:sel_h}, we restart the algorithm with $n_3$ replaced by $n_3+1$. Since the collection $\rr{S}$ of tiles is finite, the procedure will end after a finite number of steps.

\subsection*{Step: Estimates for the trilinear form}~\\
We denote $\ds \ii{I}^{n_1, n_2, n_3}:=\ii{I}_1^{n_1} \cap \ii{I}_2^{n_2} \cap \ii{I}_3^{n_3}$. This will also be a collection of dyadic intervals, and if $I_0 \in \ii{I}^{n_1, n_2, n_3}$, then $I_0=I_1 \cap I_2 \cap I_3$, with $I_j \in \ii{I}_j^{n_j}$. Set $\ds \rr{S}_{n_1, n_2, n_3}(I_0):=\rr{S}_{n_1}^1(I_1)\cap \rr{S}_{n_2}^2(I_2) \cap \rr{S}_{n_3}^3(I_3)$. We will apply the estimates of Proposition \ref{prop:generic_est} to the collection $\rr{S}_{n_1, n_2, n_3}(I_0)$. Assume $I_0$ is a fixed interval, and $n_1, n_2, n_3$ are so that 
\[
2^{-n_1} \lesssim 2^d \vert F \vert, \enskip 2^{-n_2} \lesssim 2^d \vert G \vert, \enskip 2^{-n_3} \lesssim 2^{-Md} \vert H \vert.
\]
From Proposition \ref{prop:generic_est}, we can estimate the trilinear form $\ds \Lambda_{\rr{S}_{n_1, n_2, n_3}(I_0)}$ by a product of sizes and energies. We have:
\begin{align*}
\big \vert \Lambda_{\rr{S}_{n_1, n_2, n_3}(I_0)}(f, g, h)  \big \vert & \lesssim 2^{-\frac{n_2}{r}} \cdot 2^{-n_1 \cdot 4 \alpha \theta_1} \cdot \left(2^{-n_1}\right)^{\frac{1-4 \alpha \theta_1}{2}} \cdot \vert I_0 \vert^{\frac{1-4 \alpha \theta_1}{2}} \cdot 2^{-n_2 \cdot 4 \alpha \theta_2} \cdot \left( 2^{-n_2} \right)^{\frac{2 \alpha -4 \alpha \theta_2}{2}} \\ & \quad \cdot \vert I_0 \vert^{\frac{2 \alpha -4 \alpha \theta_2}{2}}\cdot \left(2^{-n_3}\right)^{2 \alpha \theta_3} \cdot \left( 2^{-n_3} \right)^{\frac{1}{r'}-2 \alpha \theta_3} \cdot \vert I_0 \vert^{\frac{1}{r'}-2 \alpha \theta_3}\\
&+2^{-\frac{n_1}{r}} 2^{-n_1 \cdot 4 \alpha \beta_1} \cdot \left( 2^{-n_1} \right)^{\frac{2 \alpha-4 \alpha \beta_1}{2}} \cdot \vert I_0 \vert^{\frac{2 \alpha-4 \alpha \beta_1}{2}} 2^{-n_2 \cdot 4 \alpha \beta_2} \cdot \left(2^{-n_2}\right)^{\frac{1-4 \alpha \beta_2}{2}}\\
& \quad \cdot \vert I_0 \vert^{\frac{1-4 \alpha \beta_2}{2}} \cdot \left(2^{-n_3}\right)^{2 \alpha \beta_3} \cdot \left( 2^{-n_3} \right)^{\frac{1}{r'}-2 \alpha \beta_3} \cdot \vert I_0 \vert^{\frac{1}{r'}-2 \alpha \beta_3}.
\end{align*}
Eventually, after setting $\theta_j=\beta_j$, the expression above can be rewritten as
\begin{equation}
\label{eq:will_decrease}
\big \vert \Lambda_{\rr{S}_{n_1, n_2, n_3}(I_0)}(f, g, h)  \big \vert \lesssim 2^{-n_1 \left(\frac{1}{2}+2 \alpha \theta_1\right)} \cdot 2^{-n_2\left(\frac{1}{2}+2 \alpha \theta_2\right)} 2^{-n_3\cdot \frac{1}{r'}} \cdot \vert I_0 \vert.
\end{equation}
Recall that $\ii{I}_1^{n_1}$ is nonempty only as long as $\ds 2^{-n_1} \lesssim \min \left( 1, 2^d \vert  F \vert \right)$. Similarly, for $\ii{I}_2^{n_2}$ and $\ii{I}_3^{n_3}$ to be nonempty, we need to have
\begin{equation}
\label{eq:restr_n_j}
2^{-n_2} \lesssim \min \left( 1, 2^d \vert  G \vert \right), \enskip \text{and  } 2^{-n_3} \lesssim \min \left( 1, 2^{-Md}\vert  H \vert \right) \enskip\text{respectively.}
\end{equation}
With this observation, \eqref{eq:will_decrease} becomes
\begin{equation*}
\big \vert \Lambda_{\rr{S}_{n_1, n_2, n_3}(I_0)}(f, g, h)  \big \vert \lesssim 2^{-n_1 \nu_1} \cdot 2^{-n_2 \nu_2} 2^{-n_3\cdot \frac{1}{r'}} \cdot \vert I_0 \vert,
\end{equation*}
as soon as $\ds 0 \leq \nu_1 \leq \frac{1}{2}+2 \alpha \theta_1, \enskip 0 \leq \nu_2 \leq \frac{1}{2}+2 \alpha \theta_2$.

Following, we sum over intervals $\ds I_0 \in \ii{I}_1^{n_1}\cap\ii{I}_2^{n_2}\cap \ii{I}_2^{n_3}$. We have the estimates
\begin{equation}
\label{eq:sum_Is}
\sum_{I_0 \in \ii{I}^{n_1, n_2, n_3}} \vert I_0 \vert \lesssim \sum_{I \in \ii{I}_j^{n_j}} \vert I \vert \lesssim  \min \left\{2^{n_1} \vert F \vert, \enskip 2^{n_2} \vert G \vert, \enskip 2^{n_3} \vert H \vert\right\},
\end{equation}
which in turn imply (by taking the geometric average)
\[
\sum_{I_0 \in \ii{I}^{n_1, n_2, n_3}} \vert I_0 \vert \lesssim \left( 2^{n_1} \vert F \vert \right)^{\gamma_1} \cdot \left( 2^{n_2} \vert G \vert \right)^{\gamma_2} \cdot \left( 2^{n_3} \vert H \vert \right)^{\gamma_3},
\]
where $0 \leq \gamma_j \leq 1$, and $\gamma_1+\gamma_2+\gamma_3=1$.
The inequalities in \eqref{eq:sum_Is} follow from the fact that, for a fixed $n_j$,  the intervals $\ds I \in \ii{I}_j^{n_j}$ are disjoint (they were chosen to be maximal), and moreover
\[
\bigcup_{I \in \ii{I}_1^{n_1}} I \subseteq \lbrace \mathcal{M}\left( 1_F \right) \geq 2^{-n_1}\rbrace, \enskip \bigcup_{I \in \ii{I}_2^{n_2}} I \subseteq \lbrace \mathcal{M}\left( 1_G \right) \geq 2^{-n_2}\rbrace, \enskip \bigcup_{I \in \ii{I}_3^{n_3}} I \subseteq \lbrace \mathcal{M}\left( 1_{H'} \right) \geq 2^{-n_3}\rbrace.
\]
In this way, we obtain
\begin{align*}
&\big \vert \Lambda_{\rr{S}_{n_1, n_2, n_3}(I_0)}(f, g, h)  \big \vert \lesssim 2^{-n_1 \left(\nu_1- \gamma_1\right)}  2^{-n_2 \left(\nu_2- \gamma_2\right)}  \cdot 2^{-n_3 \left(\frac{1}{r'}-\gamma_3\right)} \cdot \vert F \vert^{\gamma_1} \cdot \vert G \vert^{\gamma_2} \cdot \vert H'\vert^{\gamma_3}.
\end{align*}
At this point, we are left with summing these expressions in $n_1, n_2, n_3$. The conditions in \eqref{eq:restr_n_j} will yield that
\begin{equation*}
\big \vert \Lambda_{\rr{S}_d} \left( f, g, h \right) \big \vert \lesssim \left( 2^d \vert F \vert \right)^{\nu_1 -\gamma_1} \cdot \left( 2^d \vert G \vert \right)^{\nu_2 -\gamma_2} \cdot \left( 2^{-Md} \vert H \vert \right)^{\frac{1}{r'} -\gamma_3} \cdot \vert F \vert^{\gamma_1} \cdot \vert G \vert^{\gamma_2} \cdot \vert H'\vert^{\gamma_3},
\end{equation*}
provided that, for some $\gamma_1, \gamma_2, \gamma_3$, we have
 $$\nu_1-\gamma_1 >0,\enskip  \nu_2-\gamma_2 >0, \frac{1}{r'} -\gamma_3.$$
This last condition becomes equivalent to $\nu_1+\nu_2+\dfrac{1}{r'}>1$. In this case, we obtain for $M$ suitable large
\begin{equation}
\label{eq:wanted_restr_type}
\big \vert \Lambda_{\rr{S}_d} \left( f, g, h \right) \big \vert \lesssim 2^{-10 d} \vert F \vert^{\nu_1} \cdot \vert G \vert^{\nu_2},
\end{equation}
which implies that $\Lambda_{\rr{S}}$ is of generalized restricted type $\ds \left(\nu_1, \nu_2, \nu_3  \right)$ for any admissible triple $\nu_1, \nu_2, \nu_3$ which satisfies $\nu_1+\nu_2+\nu_3=1$ and 
\[
0 < \nu_1 <\frac{1}{2}+2 \alpha \theta_1 \leq \frac{1}{r'}, \enskip 0 < \nu_2 <\frac{1}{2}+2 \alpha \theta_2 \leq \frac{1}{r'}, -1 < \nu_3<\frac{1}{r'}.
\]

Interpolation theory then yields the strong type estimates: $T_r$ maps $\ds L^p \times L^q$ into $L^s$ for any $\ds r'<p, q <\infty$, $\ds \frac{r'}{2}<s <r$.

We note that the same reasoning allows us to obtain similar estimates, if $f \in L^2 \cap L^\infty$, or $g \in L^2 \cap L^\infty$. Usually the cases $f \in L^\infty$ or $g \in L^\infty$ are obtained through duality arguments, and we illustrate this in the subsequent remark. What is interesting is that, by transforming the energy into an average (which is possible because of the stopping time), we can prove generalized restricted type estimates for the \emph{bilinear form} obtained by fixing $g \in  L^2 \cap L^\infty$.  More exactly, we can show that for such a fixed function $g$, and for sets of finite measure $F$ and $H$, with $\vert H \vert=1$, one can find a major subset $H' \subseteq H$ (which will not depend on $g$) so that, for any $\vert f \vert \leq \one_F$ and any $\ds \left(\sum_{\omega \in \Omega} \vert h_\omega \vert^{r'} \right)^{1/{r'}} \leq \one_{H'}$,
\begin{equation}
\Lambda_{\rr{S}}(f, g, h) \lesssim \|g\|_{\infty} \vert F\vert^{\nu_1},
\end{equation}
whenever $\ds \frac{1}{r} < \nu_1 < \frac{1}{r'}$. Interpolation theory implies that 
\[
\|T_r(f, g)\|_p \lesssim \|f\|_p \cdot \|g\|_{\infty}, \enskip \text{for any } \enskip r'< p <r. 
\]
\end{proof}
\begin{remark}
\label{remark_2}
An alternative way of obtaining the same range for $T_r$ is by examining the adjoint operators $T_r^{\ast, 1}$ and $T_r^{\ast, 2}$. These are defined so that
\begin{equation}
\label{eq:dual-op}
\Lambda_{\rr{S}}(f, g, h):=\langle T_r(f, g), h \rangle =\langle T_r^{\ast, 1}(h, g), f \rangle =\langle T_r^{\ast, 2}(f, h), g \rangle.
\end{equation}

Using Proposition \ref{prop:generic_est}, and the usual decomposition $\ds \rr{S}=\cup_{d \geq 0} \rr{S}_d$, where 
\[
\rr{S}_d:= \lbrace s \in \rr{S} : 1+\frac{\dist (\mathcal{E}^c, I_s)}{\vert I_s \vert} \sim 2^d  \rbrace,
\]
we can prove the following:
\begin{enumerate}
\item[(i)] $\ds T_r : L^{p} \times L^q \to L^{s}\left( \ell^r \right)$, for any $p, q, s$ satisfying
\[
\frac{1}{r} < \frac{1}{p} < \frac{1}{r'}, \enskip \text{and} \enskip   \frac{1}{r} < \frac{1}{q} < \frac{1}{r'}.
\]
\item[(ii)]$\ds T_r^{\ast, 1} : L^{s'}(\ell^{r'}) \times L^q \to L^{p'}$, for any $p, q, s$ satisfying
\[
\frac{1}{r} <\frac{1}{q} < \frac{1}{r'}, \enskip \text{and} \enskip  0 <  \frac{1}{s'} < \frac{1}{r'}.
\]
\item[(iii)] $\ds T_r^{\ast, 2} : L^p \times  L^{s'}( \ell^{r'}) \to L^{q'}$, for any $p, q, s$ satisfying
\[
\frac{1}{r} <\frac{1}{p} <\frac{1}{r'}, \enskip \text{and} \enskip  0 < \frac{1}{s'} < \frac{1}{r'}.
\]
\end{enumerate}

In this way, we obtain that the trilinear form is of generalized restricted type $\left( \nu_1, \nu_2, \nu_3 \right)$ for any triple contained inside the region $\lbrace x+y+z=1, -\frac{2}{r'} \leq x, y, z \leq \frac{1}{r'}  \rbrace$.

In particular, this implies that $T_r :L^p \times L^q \to L^s$ for any $p, q, s$ satisfying
\[
\frac{1}{p}+\frac{1}{q}=\frac{1}{s}, \enskip r' < p, q \leq \infty, \enskip \text{ and} \enskip \frac{r'}{2}< s <r. 
\]
\end{remark}

\section{An application to generalized Bochner-Riesz bilinear multiplier for rough domains}
\label{sec:an-application}
Consider ${\mathcal O}$ a bounded open subset of ${\mathbb R}^{2d}$, whose boundary has Hausdorff dimension $2d-1$. We can ask the following question:
\bigskip

{\bf Question :} What is the best non-increasing function $\ds \phi: \left[0, \text{diam}(\mathcal{O})\right] \to \left[0, \infty\right)$ such that every bilinear symbol $m$ supported on $\overline{{\mathcal O}}$ and satisfying 
\begin{equation} \left|\partial_{\xi,\eta}^\alpha m(\xi,\eta)\right| \lesssim d\big((\xi,\eta),{\mathcal O}^c\big)^{-|\alpha|} \phi\big(d((\xi,\eta),{\mathcal O}^c)\big) \enskip \text{for all  } (\xi, \eta) \in \mathcal{O}, \label{condition-symbol}
\end{equation}
and for sufficiently many multi-indices $\alpha$, gives rise to a bilinear Fourier multiplier bounded from $L^p \times L^q$ to $L^s$ ? Here the triple $(p, q, s)$ satisfies the usual H\"{o}lder scaling condition.

For some specific situations, we have some definite (sometimes almost optimal) answer (the disc and more generally the ball \cite{bil-Bochner_Riesz}, the unit cubes and any polygons ...)

\begin{proposition} \label{prop:apl-BR} Consider ${\mathcal O}$ an arbitrary bounded open subset, whose boundary has Hausdorff dimension $2d-1$. Let $r>2$ and $(p, q, s)$ a triple as in Theorem \ref{thm:main}. If $\phi$ is given by
\begin{equation}
\label{eq:apl:cond}
 \phi(t) = t^{\frac{2d-1}{r'}} (1+\log(t))^{-(1/r'+\epsilon)}
\end{equation}
for some $\epsilon>0$, then any bilinear symbol $m$ satisfying \eqref{condition-symbol} gives rise to a bilinear Fourier multiplier bounded from $L^p \times L^q$ to $L^s$.
\end{proposition}

\begin{proof} Let $\Omega$ be a Whitney covering of ${\mathcal O}$. For every integer $n$ such that $2^{-n}\leq \textrm{diam}({\mathcal O})$ denote $\Omega_n$ the subcollection of square $\omega \in \Omega$ with $\ds  2^{-n} \leq d(\omega,{\mathcal O}^c) <2^{-n+1}$.

Consider $(\chi_\omega)_{\omega\in \Omega}$ a smooth partition of the unity, associated with the Whitney covering, so that (in terms of bilinear symbols)
$$ m = \sum_{n} \sum_{\omega\in \Omega_n} m \chi_{\omega}.$$
By assumption, the symbol $m \chi_{\omega}$ satisfies
$$ \left|\partial_{\xi,\eta}^\alpha m \chi_{\omega_n}\right| \lesssim 2^{n |\alpha|} 2^{-n \frac{d-1}{r'}} n^{-(1/r'+\epsilon)}.$$
So let us renormalize them and consider for $\omega\in \Omega_n$
$$ \chi_{\omega}:= 2^{n \frac{d-1}{r'}} n ^{(1+\epsilon)} m \chi_{\omega}.$$
The operator $T_m$ (which is the bilinear Fourier multiplier associated with the symbol $m$) becomes
$$ T_{m} = \sum_{n} 2^{-n \frac{d-1}{r'}} n ^{-(1+\epsilon)} \sum_{\omega\in \Omega_n} T_{\chi_{\omega}},$$
and then we conclude by Minkoswki's inequality that
$$ \left|T_{m}(f,g)\right| \leq \left( \sum_{\omega \in \Omega} \left|T_{\chi_\omega}(f,g)\right|^r \right)^{1/r} \left( \sum_{n} 2^{-n(2d-1)} n ^{-1-r'\epsilon} (\sharp \Omega_n) \right)^{1/r'}.$$
Since ${\mathcal O}$ is supposed to have a boundary of Hausdorff dimension $(2d-1)$ we deduce that $\ds (\sharp \Omega_n) \lesssim 2^{n(2d-1)}$, which implies 
$$ \left|T_{m}(f,g)\right| \leq \left( \sum_{\omega \in \Omega} \left|T_{\chi_\omega}(f,g)\right|^r \right)^{1/r}.$$
The $\ell^r$-functional fits into the case studied in Theorem \ref{thm:main}, and hence we can infer the boundedness of $T_m$.
\end{proof}

\begin{remark}
\begin{itemize}
\item[a)] An easy observation is that $\phi(t)= t^{\frac{2d-1}{r'}} (1+\log(t))^{-(1+\epsilon)}$ is sufficient.
Indeed, we can work at a fixed scale: for every $n$,
$$ \left( \sum_{\omega \in \Omega_n} \left|T_{\chi_\omega}(f,g)\right|^r \right)^{1/r}$$
is (easily) uniformly (with respect to $n$) bounded since here we work with only one scale (it's indeed simpler than \cite{bilSqF-smooth-Fr}). We can then sum these estimates since the extra term $(1+\log(t))^{-(1+\epsilon)}$ gives a $n^{-1-\epsilon}$ decay which allows us to sum with respect to $n$. \\
Hence in this situation (dealing with an arbitrary subset ${\mathcal O}$ which may be very rough), we manage to slightly  weaken the condition on $\phi(\cdot)$ (by decreasing the order of vanishing of the symbol at the boundary) in \eqref{condition-symbol}.

\item[b)] If $p,q \geq 2$ then the previous reasoning still holds with 
$$\phi(t)= t^{\frac{2d-1}{2}} (1+\log(t))^{-(1+\epsilon)}$$
which is weaker than the condition \eqref{eq:apl:cond} in Proposition \ref{prop:apl-BR}. So the improvement is only interesting outside the local-$L^2$ range, when one of $p$ or $q$ is less than $2$.
\end{itemize}
\end{remark}

\bibliographystyle{plain}   
\bibliography{harmonic.bib}

\begin{thebibliography}{10}

\bibitem{vv_BHT}
Cristina Benea and Camil Muscalu.
\newblock Multiple vector valued inequalities via the helicoidal method.
\newblock \url{http://arxiv.org/pdf/1511.04948v1.pdf}.

\bibitem{bilinLP}
Fr\'ed\'eric Bernicot.
\newblock ${L}^p$ boundedness for nonsmooth bilinear {L}ittlewood-{P}aley
  square functions.
\newblock {\em Math. Ann.}, 2351(1):1--49, 2011.

\bibitem{bil-Bochner_Riesz}
Fr{\'e}d{\'e}ric Bernicot, Loukas Grafakos, Liang Song, and Lixin Yan.
\newblock The bilinear {B}ochner-{R}iesz problem.
\newblock {\em J. Anal. Math.}, 127:179--217, 2015.

\bibitem{bilSqF-smooth-Fr}
Fr{\'e}d{\'e}ric Bernicot and Saurabh Shrivastava.
\newblock Boundedness of smooth bilinear square functions and applications to
  some bilinear pseudo-differential operators.
\newblock {\em Indiana Univ. Math. J.}, 60(1):233--268, 2011.

\bibitem{initial_Carleson}
Lennart Carleson.
\newblock On convergence and growth of partial sums of {F}ourier series.
\newblock {\em Acta Math.}, 116:135--157, 1966.

\bibitem{CoifMeyer-ondelettes}
Ronald Coifman and Yves Meyer.
\newblock {\em Wavelets}, volume~48 of {\em Cambridge Studies in Advanced
  Mathematics}.
\newblock Cambridge University Press, Cambridge, 1997.
\newblock Calder{\'o}n-Zygmund and multilinear operators, Translated from the
  1990 and 1991 French originals by David Salinger.

\bibitem{DiestelRemarksBilLP}
Geoff Diestel.
\newblock Some remarks on bilinear {L}ittlewood-{P}aley theory.
\newblock {\em J. Math. Anal. Appl.}, 307(1):102--119, 2005.

\bibitem{DiestelGraf-maxBilOpDIlations}
Geoff Diestel and Loukas Grafakos.
\newblock Maximal bilinear singular integral operators associated with
  dilations of planar sets.
\newblock {\em J. Math. Anal. Appl.}, 332(2):1482--1494, 2007.

\bibitem{bilinear_disc_multiplier}
Loukas Grafakos and Xiaochun Li.
\newblock The disc as a bilinear multiplier.
\newblock {\em Amer. J. Math.}, 128(1):91--119, 2006.

\bibitem{carleson-hunt}
Richard~A. Hunt.
\newblock On the convergence of {F}ourier series.
\newblock In {\em Orthogonal {E}xpansions and their {C}ontinuous {A}nalogues
  ({P}roc. {C}onf., {E}dwardsville, {I}ll., 1967)}, pages 235--255. Southern
  Illinois Univ. Press, Carbondale, Ill., 1968.

\bibitem{JournCZopRF2}
Jean-Lin Journ{\'e}.
\newblock Calder{\'o}n-{Z}ygmund operators on product spaces.
\newblock {\em Revista Matemática Iberoamericana}, 1(3):55--91, 1985.

\bibitem{bilinear-LP-Lacey}
Michael Lacey.
\newblock On bilinear {L}ittlewood-{P}aley square functions.
\newblock {\em Publ. Mat.}, 40(2):387--396, 1996.

\bibitem{initial_BHT_paper}
Michael Lacey and Christoph Thiele.
\newblock On {C}alder\'on's conjecture.
\newblock {\em Ann. of Math. (2)}, 149(2):475--496, 1999.

\bibitem{smoothEDbilineairRF}
Parasar Mohanty and Saurabh Shrivastava.
\newblock A note on the bilinear {L}ittlewood-{P}aley square function.
\newblock {\em Proc. Amer. Math. Soc.}, 138(6):2095--2098, 2010.

\bibitem{CamilPhDthesis}
Camil Muscalu.
\newblock Phd thesis.
\newblock {\em Brown University}, 2000.

\bibitem{multilinear_harmonic}
Camil Muscalu and Wilhem Schlag.
\newblock {\em Classical and {Multilinear} {Harmonic} {Analysis}}.
\newblock Cambridge University Press, 2013.

\bibitem{multilinearMTT}
Camil Muscalu, Terence Tao, and Christoph Thiele.
\newblock Multi-linear operators given by singular multipliers.
\newblock {\em J. Amer. Math. Soc.}, 15(2):469--496, 2002.

\bibitem{biest}
Camil Muscalu, Terence Tao, and Christoph Thiele.
\newblock {$L^p$} estimates for the biest. {II}. {T}he {F}ourier case.
\newblock {\em Math. Ann.}, 329(3):427--461, 2004.

\bibitem{RF}
Jos{\'e} Rubio~de Francia.
\newblock A {Littlewood}-{Paley} {Inequality} for {Arbitrary} {Intervals}.
\newblock {\em Revista Matematica Iberoamericana}, 1(2):891--921, 1985.

\bibitem{wave_packet}
Christoph Thiele.
\newblock {\em Wave packet analysis}, volume 105 of {\em CBMS Regional
  Conference Series in Mathematics}.
\newblock Published for the Conference Board of the Mathematical Sciences,
  Washington, DC; by the American Mathematical Society, Providence, RI, 2006.

\end{thebibliography}

\end{document}